\documentclass[11pt,letterpaper]{amsart} \textwidth=14.5cm \oddsidemargin=1cm
\usepackage[rgb]{xcolor}
\usepackage{tikz}
\usepackage{verbatim}
\usepackage{amsmath}
\usepackage{amsxtra}
\usepackage{amscd}
\usepackage{amsthm}
\usepackage{amsfonts}
\usepackage{amssymb}
\usepackage{eucal}
\usepackage{float}
\usepackage{mathrsfs}
\usepackage{graphicx}
\usepackage{stmaryrd}
\usepackage{tikz-cd}
\usetikzlibrary{arrows}
\usetikzlibrary{cd}
\usepackage{ytableau}
\usepackage{tikz-cd}
\usetikzlibrary{arrows}
\usetikzlibrary{cd}
\tikzcdset{every label/.append style = {font = \small}}

\usepackage{latexsym,todonotes}

\PassOptionsToPackage{hyphens}{url}\usepackage[linktocpage=true]{hyperref}
\hypersetup{colorlinks,linkcolor=blue,urlcolor=orange,citecolor=blue}
\usepackage[all, cmtip]{xy}

\usepackage{stmaryrd}
\usepackage{color} 

\usepackage{amsrefs}

\newtheorem{thm}{Theorem} [section]
\newtheorem{thm2}[thm]{Theorem$^\prime$}

\newtheorem{lem}[thm]{Lemma}
\newtheorem{prop}[thm]{Proposition}

\newtheorem{claim}[thm]{Expectation}

\theoremstyle{definition}
\newtheorem{definition}[thm]{Definition}
\newtheorem{example}[thm]{Example}

\theoremstyle{remark}
\newtheorem{rem}[thm]{Remark}

\numberwithin{equation}{section}

\begin{document}

\newcommand{\thmref}[1]{Theorem~\ref{#1}}
\newcommand{\secref}[1]{Section~\ref{#1}}
\newcommand{\lemref}[1]{Lemma~\ref{#1}}
\newcommand{\propref}[1]{Proposition~\ref{#1}}
\newcommand{\corref}[1]{Corollary~\ref{#1}}
\newcommand{\remref}[1]{Remark~\ref{#1}}
\newcommand{\eqnref}[1]{(\ref{#1})}

\newcommand{\exref}[1]{Example~\ref{#1}}

 \newcommand{\GSp}{\mathrm{GSp}}
 \newcommand{\PGSp}{\mathrm{PGSp}}
\newcommand{\PGSO}{\mathrm{PGSO}}
\newcommand{\PGO}{\mathrm{PGO}}
\newcommand{\SO}{\mathrm{SO}}
\newcommand{\GO}{\mathrm{GO}}
\newcommand{\GSO}{\mathrm{GSO}}
\newcommand{\Spin}{\mathrm{Spin}}
\newcommand{\Sp}{\mathrm{Sp}}
\newcommand{\PGL}{\mathrm{PGL}}
\newcommand{\GL}{\mathrm{GL}}
\newcommand{\SL}{\mathrm{SL}}
\newcommand{\U}{\mathrm{U}}
\newcommand{\ind}{\mathrm{ind}}
\newcommand{\Ind}{\mathrm{Ind}}
 \newcommand{\triv}{\mathrm{triv}}
 
\newtheorem{innercustomthm}{{\bf Theorem}}
\newenvironment{customthm}[1]
  {\renewcommand\theinnercustomthm{#1}\innercustomthm}
  {\endinnercustomthm}
  
  \newtheorem{innercustomcor}{{\bf Corollary}}
\newenvironment{customcor}[1]
  {\renewcommand\theinnercustomcor{#1}\innercustomcor}
  {\endinnercustomthm}
  
  \newtheorem{innercustomprop}{{\bf Proposition}}
\newenvironment{customprop}[1]
  {\renewcommand\theinnercustomprop{#1}\innercustomprop}
  {\endinnercustomthm}

\newcommand{\bbinom}[2]{\begin{bmatrix}#1 \\ #2\end{bmatrix}}
\newcommand{\cbinom}[2]{\set{\^!\^!\^!\begin{array}{c} #1 \\ #2\end{array}\^!\^!\^!}}
\newcommand{\abinom}[2]{\ang{\^!\^!\^!\begin{array}{c} #1 \\ #2\end{array}\^!\^!\^!}}
\newcommand{\qfact}[1]{[#1]^^!}

\newcommand{\nc}{\newcommand}

\nc{\Ord}{\text{Ord}_v}

 \nc{\A}{\mathcal A} 
  \nc{\G}{\mathbb G} 
\nc{\Ainv}{\A^{\rm inv}}
\nc{\aA}{{}_\A}
\nc{\aAp}{{}_\A'}
\nc{\aff}{{}_\A\f}
\nc{\aL}{{}_\A L}
\nc{\aM}{{}_\A M}
\nc{\Bin}{B_i^{(n)}}
\nc{\dL}{{}^\omega L}
\nc{\Z}{{\mathbb Z}}
 \nc{\C}{{\mathbb C}}
 \nc{\N}{{\mathbb N}}
 \nc{\R}{{\mathbb R}}
 \nc{\fZ}{{\mf Z}}
 \nc{\F}{{\mf F}}
 \nc{\Q}{\mathbb{Q}}
 \nc{\la}{\lambda}
 \nc{\ep}{\epsilon}
 \nc{\h}{\mathfrak h}
 \nc{\He}{\bold{H}}
 \nc{\htt}{\text{tr }}
 \nc{\n}{\mf n}
 \nc{\g}{{\mathfrak g}}
 \nc{\DG}{\widetilde{\mathfrak g}}
 \nc{\SG}{\breve{\mathfrak g}}
 \nc{\is}{{\mathbf i}}
 \nc{\V}{\mf V}
 \nc{\bi}{\bibitem}
 \nc{\E}{\mc E}
 \nc{\ba}{\tilde{\pa}}
 \nc{\half}{\frac{1}{2}}
 \nc{\hgt}{\text{ht}}
 \nc{\ka}{\kappa}
 \nc{\mc}{\mathcal}
 \nc{\mf}{\mathfrak} 
 \nc{\hf}{\frac{1}{2}}
\nc{\ov}{\overline}
\nc{\ul}{\underline}
\nc{\I}{\mathbb{I}}
\nc{\xx}{{\mf x}}
\nc{\id}{\text{id}}
\nc{\one}{\bold{1}}
\nc{\mfsl}{\mf{sl}}
\nc{\mfgl}{\mf{gl}}
\nc{\ti}[1]{\textit{#1}}
\nc{\Hom}{\mathrm{Hom}}
\nc{\Irr}{\mathrm{Irr}}
\nc{\Cat}{\mathscr{C}}
\nc{\CatO}{\mathscr{O}}
\renewcommand{\O}{\mathrm{O}}
\nc{\Tan}{\mathscr{T}}
\nc{\Umod}{\mathscr{U}}
\nc{\Func}{\mathscr{F}}
\nc{\Kh}{\text{Kh}}
\nc{\Khb}[1]{\llbracket #1 \rrbracket}

\nc{\ua}{\mf{u}}
\nc{\nb}{u}
\nc{\inv}{\theta}
\nc{\mA}{\mathcal{A}}
\newcommand{\TT}{\mathbf T}
\newcommand{\TA}{{}_\A{\TT}}
\newcommand{\tK}{\widetilde{K}}
\newcommand{\al}{\alpha}
\newcommand{\Fr}{\bold{Fr}}

\nc{\Qq}{\Q(v)}
\nc{\uu}{\mathfrak{u}}
\nc{\Udot}{\dot{\U}}

\nc{\f}{\bold{f}}
\nc{\fprime}{\bold{'f}}
\nc{\B}{\bold{B}}
\nc{\Bdot}{\dot{\B}}
\nc{\Dupsilon}{\Upsilon^{\vartriangle}}
\newcommand{\T}{\texttt T}
\newcommand{\vs}{\varsigma}
\newcommand{\Pa}{{\bf{P}}}
\newcommand{\Padot}{\dot{\bf{P}}}

\nc{\ipsi}{\psi_{\imath}}
\nc{\Ui}{{\bold{U}^{\imath}}}
\nc{\uidot}{\dot{\mathfrak{u}}^{\imath}}
\nc{\Uidot}{\dot{\bold{U}}^{\imath}}
 \nc{\be}{e}
 \nc{\bff}{f}
 \nc{\bk}{k}
 \nc{\bt}{t}
 \nc{\bs}{\backslash}
 \nc{\BLambda}{{\Lambda_{\inv}}}
\nc{\Ktilde}{\widetilde{K}}
\nc{\bktilde}{\widetilde{k}}
\nc{\Yi}{Y^{w_0}}
\nc{\bunlambda}{\Lambda^\imath}
\newcommand{\Iwhite}{\I_{\circ}}
\nc{\ile}{\le_\imath}
\nc{\il}{<_{\imath}}

\newcommand{\ff}{B}


\nc{\etab}{\eta^{\bullet}}
\newcommand{\Iblack}{\I_{\bullet}}
\newcommand{\wb}{w_\bullet}
\newcommand{\UIblack}{\U_{\Iblack}}

\newcommand{\blue}[1]{{\color{blue}#1}}
\newcommand{\red}[1]{{\color{red}#1}}
\newcommand{\green}[1]{{\color{green}#1}}
\newcommand{\white}[1]{{\color{white}#1}}

\newcommand{\dvd}[1]{t_{\odd}^{{(#1)}}}
\newcommand{\dvp}[1]{t_{\ev}^{{(#1)}}}
\newcommand{\ev}{\mathrm{ev}}
\newcommand{\odd}{\mathrm{odd}}

\newcommand\TikCircle[1][2.5]{{\mathop{\tikz[baseline=-#1]{\draw[thick](0,0)circle[radius=#1mm];}}}}

\newcommand{\commentcustom}[1]{}

\raggedbottom

\title[Generalised Whittaker models and relative Langlands duality]
{Generalised Whittaker models as instances of relative Langlands duality}

\author{Wee Teck Gan}
\author{Bryan Wang Peng Jun}
 \address{Department of Mathematics, National University of Singapore, 10 Lower Kent Ridge Road, Singapore 119076}
\email{matgwt@nus.edu.sg}
\email{bwang@u.nus.edu}

\begin{abstract}
The recent proposal by Ben-Zvi, Sakellaridis and Venkatesh of a duality in the relative Langlands program, leads, via the process of quantization of Hamiltonian varieties, to a duality theory of branching problems. This often unexpectedly relates two \textit{a priori} unrelated branching problems. We examine how the generalised Whittaker (or Gelfand-Graev) models serve as the prototypical example for such branching problems. We give a characterization, for the orthogonal and symplectic groups, of the generalised Whittaker models possibly contained in this duality theory. We then exhibit an infinite family of examples of this duality, which, provably at the local level via the theta correspondence, satisfy the conjectural expectations of duality.
\end{abstract}

\maketitle

\setcounter{tocdepth}{1}
\tableofcontents

\section{Introduction}

One of the central themes of the relative Langlands program is to characterize the non-vanishing of certain periods of automorphic forms on a reductive group $G$ (relative to a subgroup $H$ of $G$), and when the period is non-zero, to relate it  to certain (special values of) automorphic L-functions. The corresponding problem at the local level is the $H$-distinction problem: the classification of irreducible smooth representations of $G$ which possess  nonzero $H$-invariant functionals. This $H$-distinction problem is
 equivalent, by Frobenius reciprocity, to the classification of the irreducible $G$-submodules of $C^{\infty}(X)$ with $X = H\backslash G$; we shall call these the $X$-distinguished representations. One may also consider the corresponding $L^2$-problem, in which case one is interested in determining the spectral decomposition of the unitary representation $L^2(X)$. In any of these settings, the expectation is that the $X$-distinguished representations are obtained as Langlands functorial lifts from another (typically smaller) group. 
 \vskip 5pt
 
 \subsection{\bf Spherical varieties and \cite{SV}}
 More precise predictions for these problems were laid out in \cite{SV} in the case when $H$ is a spherical subgroup of $G$, so that
$X= H\bs G$ is a ($G$-homogeneous) \textit{spherical} variety. The main reason for singling out spherical subgroups is the expectation that the spectral decomposition in question will be multiplicity-free (or at least has finite multiplicities). In the spirit of the Langlands philosophy, \cite{SV} associates to the spherical variety $X$ the following dual data:
\vskip 5pt
\begin{itemize}
\item a Langlands dual group $X^{\vee}$ and a map
\[  \iota_X: X^{\vee} \times {\rm SL}_2(\C) \longrightarrow G^{\vee}.\]
\vskip 5pt

\item a (graded) finite-dimensional (typically) symplectic representation $V_X$ of $X^{\vee}$.
\end{itemize}
It was then conjectured that the $X$-distinguished representations (of Arthur type) are those whose A-parameters factor through the map $\iota_X$, thus making precise the group from which the Langlands functorial lifting originates. The representation $V_X$ of $X^{\vee}$ is the main ingredient allowing one to form the automorphic L-function which controls the relevant period. For a more detailed discussion of this, the reader can consult \cite{SV} or the introduction of \cite{GaWa}. 
\vskip 5pt

There are however plenty of important examples which appear to fit into the framework of the relative Langlands program, but which do not arise from the usual class of spherical varieties. One example is the Bessel and Fourier-Jacobi models \cite{GGP}, or more generally, generalised Whittaker models with a Whittaker-twisted component arising from a nilpotent orbit in $G$. Another example is Howe duality (or theta correspondence) \cite{Sa2}, concerning the decomposition of the Weil representation which may be thought of as arising from a symplectic vector space.  In each of these problems, one encounters a natural $G$-module whose spectral decomposition is multiplicity-free, but the $G$-module is not of the form $C^{\infty}(X)$ with $X$ a spherical variety.  
\vskip 10pt

\subsection{\bf Hyperspherical varieties and \cite{BZSV}}
The above considerations led Ben-Zvi, Sakellaridis and Venkatesh to investigate the natural setting for the relative Langlands program.  
Partly motivated by the work of Kapustin-Witten \cite{KW} interpreting geometric Langlands duality as an electric-magnetic duality of (four-dimensional) topological quantum field theories (TQFT) and the work of Gaiotto-Witten \cite{GaWi} on the boundary conditions of these TQFT, they propose in their recent paper \cite{BZSV} a broader framework for the relative Langlands program.

\vskip 5pt

According to their new proposal, the basic objects considered by the relative Langlands program should be a class of \textit{Hamiltonian $G$-varieties} $M$ called \textit{hyperspherical varieties}. From this point of view, instead of considering spherical varieties $X$, one should consider instead its cotangent variety $M = T^*(X)$.
By the process of quantization (broadly construed), these hyperspherical $G$-varieties give rise to unitary $G$-representations whose spectral decomposition is what the relative Langlands program should be concerned with. This point of view reconnects us with the classical philosophy of (geometric) quantization as a means to study the representation theory of Lie groups, a process which itself has its roots in the development of quantum mechanics.  \vskip 5pt

We recall the basic definition of a hyperspherical $G$-variety in Section \ref{sec:RelativeLanglandsDuality}. A key result shown in \cite{BZSV} is a structure theorem for such varieties. It turns out that any hyperspherical $G$-variety can be built out of the following initial data:
\vskip 5pt

\begin{itemize}
\item a map 
\[  \iota: H \times {\rm SL}_2 \longrightarrow  G \]
with $H \subset Z_G(\iota({\rm SL}_2))$ a spherical subgroup;
\vskip 5pt

\item a finite-dimensional symplectic representation $S$ of $H$.
\end{itemize}
Given these initial data, the corresponding hyperspherical $G$-variety $M$ is built up  by the process of 
 `Whittaker induction' of the symplectic $H$-vector space $S$ along the homomorphism $\iota$, which is an instance  of Hamiltonian reduction \`a la Marsden-Weinstein.
 The quantizations of these hyperspherical $M$'s capture the aforementioned examples of  the space of smooth functions on spherical varieties,  the generalised Whittaker models  and the Weil representation.  In fact, these  account for the main prototypical examples of the quantizations of hyperspherical varieties, with the general case built up by the amalgam of these three cases.  Thus, this enlarged framework of the relative Langlands program captures all the known examples that one would like to include. 
 
 \vskip 10pt
 
 \subsection{\bf BZSV Duality}
 Observe that the initial data  
 \[  (\iota: H \times {\rm SL}_2 \rightarrow G, S) \]
 used in the construction of a hyperspherical variety is very similar to the key data 
 \[   (\iota_X: X^{\vee} \times {\rm SL}_2 \rightarrow G^{\vee}, V_X) \]
 used in the formulation of the conjecture of \cite{SV} recalled above. If one were to apply the process of Whittaker induction to the latter data, one will get a hyperspherical $G^{\vee}$-variety $M^{\vee}$ (over $\C$). 
 \vskip 5pt
 
 Now another novel realization in \cite{BZSV} is that the conjecture of \cite{SV} on the classification of $X$-distinguished representations can be elegantly reformulated in terms of $M^{\vee}$. Namely, the $X$-distinguished representations (of Arthur type) are those whose A-parameters have nonempty fixed point set on $M^{\vee}$.  This suggests that not only should the basic objects in the relative Langlands program be these hyperspherical $G$-varieties $M$'s, the dual objects describing the solution of the spectral problem arising from $M$ should also be hyperspherical varieties, for the Langlands dual group $G^{\vee}$.  
 
 \vskip 5pt
 Pursuing this train of thought further,  \cite{BZSV} suggested that there should exist an involutive theory of \textit{duality} of such hyperspherical varieties, 
\[G\circlearrowright M\longleftrightarrow M^\vee \circlearrowleft G^\vee,\]
 relating two \textit{a priori} unrelated instances of the relative Langlands program, namely the spectral problems associated to the quantizations of $M$ and $M^{\vee}$. From the viewpoint of the geometric Langlands program,
they view this proposed relative Langlands duality as a classical manifestation of  a duality of boundary conditions of the TQFT's arising in the work of Kapustin-Witten and Gaiotto-Witten. However, there is as yet no firm definition of this purported theory of duality. 

As an instance of this purported duality, the trivial $G$-space (consisting of a single point) is dual to the Whittaker (twisted) cotangent bundle of $G^{\vee}$. 
Another striking instance of duality is that the hyperspherical variety underlying the branching problem occurring in the GGP conjecture \cite{GGP}, is dual to that underlying the theta correspondence, which is just a symplectic vector space acted upon by the corresponding reductive dual pair (Remark \ref{rem:NoIsotypic2}).
 \vskip 5pt

Let us remark also that it is expected in \cite{BZSV} that the duality theory should extend to a much wider class of Hamiltonian spaces, such as non-smooth spaces, or spaces which are not varieties (e.g. stacks, or derived schemes). In this paper, we restrict ourselves to the definition of `hyperspherical' as given in \cite{BZSV}, for which there exists a reasonable structure theory and formulation of expectations for the duality. 

\vskip 10pt

\subsection{\bf The result of this paper}\label{sec:IntroResult}
 In this paper, we consider mainly the special case of a hyperspherical variety $M$ whose initial data satisfies
 \[  H = Z_G(\iota({\rm SL}_2)) \quad \text{and} \quad S = 0. \] 
 (We also consider $S\ne 0$ in some cases, namely when the associated nilpotent conjugacy class, as below, is non-even. The choice of $S$ is largely dictated by the choice of nilpotent conjugacy class.)
 \vskip 5pt
 Such a $M$ is thus determined by a homomorphism
 \[  \iota: {\rm SL}_2 \longrightarrow G \]
 which, by the Jacobson-Morozov theorem, is associated to a nilpotent conjugacy class 
 \[  e = d\iota\left( \begin{array}{cc} 0 & 1 \\
0 & 0 \end{array} \right) \in \mathfrak{g} = {\rm Lie}(G). \]
 The obtained hyperspherical variety $M_e$ can be described more explicitly as 
 \[  M_e =  ( (f + \mathfrak{g}^e) \cap \mathfrak{h}^\perp) \times^H G  \]
 where
 \[  f =  d\iota \left( \begin{array}{cc}
  0 & 0 \\
  1 & 0 \end{array} \right),  \qquad \mathfrak{g}^e = {\rm Ker}({\rm ad}(e)) \quad \text{and} \quad  \mathfrak{h} = {\rm Lie}(H). \]
  Thus, $M_e$ is built from the Slodowy slice associated to  $e$ and  
 the corresponding quantization $\Pi_e$ of $M_e$ is  an instance of the  generalized Whittaker (or Gelfand-Graev) models. 
 \vskip 5pt
 
 After this preparation, we can describe the main results of this paper:
 \vskip 5pt
 
 \begin{itemize}
 \item[(a)]  Our first result gives a characterization of those $e$'s for
  the orthogonal and symplectic groups, which could possibly give rise to hyperspherical varieties. The precise statements can be found in Theorem \ref{thm:OrthogonalClassification} and  Theorem \ref{thm:SymplecticClassification}. In fact, the same upper bound for the possible $e$'s holds, regardless of the choice of $S$.  
  \vskip 5pt
  
  For example, an infinite family of such $e$'s consists of those whose corresponding partition has associated Young diagrams of hook type.
  The corresponding generalized Whittaker models are the so-called Bessel and Fourier-Jacobi models that one encounters in the GGP conjecture.
  \vskip 5pt
  
  \item[(b)]  Our second result determines the hyperspherical dual  $M_e^{\vee}$ of $M_e$, for some of those $e$'s in (a), in particular for all those $e$'s of hook type.  The precise statements can be found in Theorem \ref{thm:EvenOrthogonalMain} and  Theorem  \ref{thm:OddOrthogonalMain}. For these $e$'s of hook type, it turns out that 
  \[  M_e^{\vee} \cong M_{e^{\vee}} \] 
  for some $e^{\vee}$ which is also of hook type.  
  \end{itemize}
  \vskip 5pt
  
  Because there is no formal definition of the duality $M \longleftrightarrow M^{\vee}$,  let us explain what we mean by (b) above. 
  \vskip 5pt
  
  In this paper, when we say that a hyperspherical $G$-variety $M$  with associated data $(\iota: H \times {\rm SL}_2 \rightarrow  G, S)$ is dual to a hyperspherical $G^{\vee}$-variety $M^{\vee}$ with associated data $(\iota^{\dagger}:  H^{\dagger} \times {\rm SL}_2 \rightarrow G^{\vee}, S^{\dagger})$, we mean that the following two statements hold:
  \vskip 5pt
  
  \begin{itemize}
  \item The irreducible representations of $G$ of Arthur type which intervene in the spectral decomposition of the quantization $\Pi_M$ of $M$ have A-parameters factoring through the map $\iota^{\dagger}$;
  \vskip 5pt
  
  \item The irreducible representations of $G^{\vee}$ of Arthur type which intervene in the spectral decomposition of the quantization $\Pi_{M^{\vee}}$ of $M^{\vee}$ have A-parameters factoring through the map $\iota$;
    \end{itemize}
  In other words, in establishing the results highlighted in (b) above, we are solving a pair of branching problems and showing that their answers can be described in terms of each other. The main tool used in our proof of (b) above is  the theta correspondence, and in particular  a result of Gomez and Zhu (\cite{GZ}, \cite{Zh}) which relates generalised Whittaker models via the theta correspondence, and is a manifestation of the more general principle that the theta correspondence often relates two periods on each member of a dual pair. 

  Finally, let us remark that there is nothing essential about the choice to focus only on orthogonal and symplectic groups in this paper; one could obtain similar results for the general linear and exceptional groups by similar methods, that is, the unitary and exceptional theta correspondences respectively. 
  

\subsection{Duality and symplectic reduction}

The proof of our main results above (using the theorems of Gomez-Zhu) has an underlying geometric interpretation, which is a manifestation of the following principle that we learned from a suggestion of Venkatesh: \[ \textit{Hyperspherical duality `commutes' with symplectic reduction.} \]

More precisely, one expects a diagram of the form \[ \begin{tikzcd}[ampersand replacement=\&, row sep=large,column sep=huge, every label/.append
style={font=\normalsize}, outer sep=5pt]
M_1' \arrow{d}[swap]{\begin{array}{@{}c@{}}\text{symplectic reduction} \\ \text{with } \{0\} \end{array}}\arrow[r,leftrightarrow,"\text{duality}"] \& M_2' \arrow[d," \begin{array}{@{}c@{}}\text{Whittaker reduction} \end{array}"] \\
M_1 \arrow[r,leftrightarrow,"\text{duality} "] \& M_2
\end{tikzcd} \]

As a somewhat trivial illustration of this principle, consider for instance the hyperspherical duality in the `group case' \cite{BZSV}, between the $(G\times G)$-variety $T^\ast G$ and the $(G^\vee\times G^\vee)$-variety $T^\ast G^\vee$ (with one factor twisted by the Chevalley involution). Now symplectic reduction of $T^\ast G$ gives the trivial $G$-variety, whose dual is the Whittaker cotangent bundle for $G^\vee$, which can be obtained by Whittaker reduction of $T^\ast G^\vee$. 

As we will explain in Section \ref{sec:DualitySymplecticReduction}, the proof of our main results above can in fact be interpreted as a quantization of this geometric principle, taking $M_1’, M_2’$ as the hyperspherical dual pair associated with the GGP problem and the theta correspondence.

\subsection{Organisation of the paper} The paper is organised as follows. In Section \ref{sec:NilpOrbits} we set up the necessary preliminaries on nilpotent orbits and associated objects, and introduce the generalised Whittaker models. After reviewing preliminaries regarding Hamiltonian spaces and quantization in Section \ref{sec:HamiltonianQuantization}, in Section \ref{sec:RelativeLanglandsDuality} we introduce the relative Langlands duality of hyperspherical varieties, and explain how the generalised Whittaker models fit into the framework of hyperspherical varieties. 

In Section \ref{sec:HypersphericalClassification} we give a characterization of the generalised Whittaker models which could possibly satisfy the hyperspherical assumption. The rest of the paper is devoted to studying the cases which arise from Section \ref{sec:HypersphericalClassification}: In Section \ref{sec:ThetaCorrespondence}, we review the theory and results surrounding the theta correspondence, in preparation for Section \ref{sec:HookType}, where we study the `hook-type' generalised Whittaker models, which provide a class of examples of hyperspherical dual pairs. In Section \ref{sec:DualitySymplecticReduction}, we consider related hyperspherical dual pairs (for the group $G\times H$) \cite{FU}, and in doing so examine how hyperspherical duality interacts with the operation of symplectic reduction. Finally, in Section \ref{sec:ExceptionalPartitions} we make some brief remarks on the exceptional partitions which arise from the characterization of Section \ref{sec:HypersphericalClassification}. 

\subsection{Notation and conventions} Throughout, let $F$ be a (non-Archimedean) local field of characteristic 0, and fix a non-trivial unitary character $\psi : F\rightarrow \C^\times$. $G$ will always denote a (split) reductive group, and we work with split forms of classical groups, unless stated otherwise. Unless otherwise specified we work throughout with smooth admissible representations. We denote by $\ind$ and $\Ind$ the (unnormalized) compact induction and induction of representations respectively. We work over $\C$ (or an algebraically closed field of characteristic zero) when dealing with hyperspherical varieties and their duals, and work over $F$ otherwise. 

 \subsection{Acknowledgements} We would like to thank Yiannis Sakellaridis and Akshay Venkatesh for illuminating discussions about \cite{BZSV} during the course of this work and a number of very helpful comments on the paper. The first author thanks Ivan Losev for pointing out the reference \cite{FU} at the Nisyros conference 2023. The second author would also like to thank Nhat Hoang Le for some helpful conversations while the work on this paper was ongoing. 
W.T. Gan is partially supported by a Singapore government MOE Tier 1 grant R-146-000-320-114 and a Tan Chin Tuan Centennial Professorship.

\section{Nilpotent orbits and generalised Whittaker models}\label{sec:NilpOrbits}

\subsection{Preliminaries on nilpotent orbits}

We first review preliminaries concerning nilpotent orbits (Dynkin-Kostant theory), mainly to fix notation and to highlight the similarities between the structure theory of hyperspherical varieties and the formation of generalised Whittaker models. A reference for the relevant theory is \cite{CM}.

Fix $\kappa$, an ${\rm Ad}(G)$-invariant non-degenerate bilinear form on $\g$. Let $\gamma=\{e,h,f\}\subset \g$ be an $\mfsl_{2}$-triple associated to a nilpotent orbit of $\g$. 

\begin{rem}
By the Jacobson-Morozov theorem (and other results of Kostant) \cite{CM}, there is a correspondence between (conjugacy classes of) $\mfsl_2$-triples and nilpotent orbits; as such, in this paper, we will often refer to $\mfsl_2$-triples and their corresponding nilpotent orbits interchangeably, with no confusion to be expected for the reader. 
\end{rem}

Under the adjoint $\mfsl_2$-action, $\g$ decomposes into $\mfsl_2$ weightspaces \[\g_j=\{v\in \g \mid {\rm ad}(h)v=jv\}\] for $j\in \Z$. We have the parabolic \[\mf p=\oplus_{j\ge 0} \g_{j} = \mf l \oplus \mf u,\] where $\mf l = \g_0$. Set $\mf u^+:=\oplus_{j\geq 2} \g_{j}$.

We get corresponding subgroups $P=L\ltimes U$ and $U^+$  of $G$. Note \[ L=\{l\in G \mid \mbox{${\rm Ad}(l)h=h$}\}\] is the stabiliser of $h$. Denote the centraliser of $\gamma$ by \[M_\gamma=\{l\in L \mid \mbox{${\rm Ad}(l)e=e$}\}=\{g\in G \mid \mbox{${\rm Ad}(g)e=e$, ${\rm Ad}(g)f=f$, ${\rm Ad}(g)h=h$}\},\] which is reductive. 

We define a character $\chi_{\gamma,\psi}$ on $U^+$ via 
\begin{equation}
\label{defchi}
\chi_{\gamma,\psi}(\exp u):=\psi (\kappa(f,u)), \ \ \ \ \forall \ u\in \mf u^+.
\end{equation}
Denote also \[ \kappa_f(u) := \kappa(f,u). \]
Since $\psi$ is fixed, in what follows we will drop the subscript and simply write $\chi_\gamma$.

\subsection{Classification of nilpotent orbits}\label{NilpPartition} Suppose now $G$ is the isometry group of a $n$-dimensional vector space $V$ equipped with an orthogonal or symplectic form $B$ over $F$. 

From an $\mfsl_2$-triple as above, we obtain an $\mfsl_2$-representation on $V$ and the decomposition $V=\oplus_{j=1}^l V^{(j)}$, where \[ V^{(j)}=W_j^{\oplus a_j} \cong W_j \otimes V_j \] is the isotypic component of $V$ for the irreducible $j$-dimensional representation $W_j$ of $\mfsl_2$, and $V_j$ is a $a_j$-dimensional multiplicity space. 

Recall, from standard $\mfsl_2$-theory, that $W_j$ is symplectic (resp. orthogonal) if $j$ is even (resp. odd); fix corresponding $\mfsl_2$-invariant forms $A_j$ on $W_j$.

The form $B$ induces a symplectic or orthogonal form $B_j$ on the $a_j$-dimensional multiplicity spaces $V_j$. $B_j$ is symplectic if $j$ is even and $B$ is orthogonal or if $j$ is odd and $B$ is symplectic, and otherwise $B_j$ is orthogonal. 

$M_\gamma$ is in fact (isomorphic to) the direct product of the isometry groups of $(V_j,B_j)$; we denote \[ M_\gamma \cong \prod_{j=1}^l G(V_j,B_j). \] 

The above furnishes a parameterisation of the nilpotent orbits in $G$, by the datum of:
\begin{itemize}
    \item the partition $\lambda=[l^{a_l},\dots, 1^{a_1}]$ of $n$, and
    \item the forms on the multiplicity spaces $(V_j,B_j)$,
\end{itemize} 
such that the $(V_j,B_j)$ are compatible with $B$ in the above way; to be precise, this means that the $B_j$ must be the forms that would be induced from $B$ as above, or that \[ \bigoplus_j (V_j, B_j) \otimes (W_j, A_j) \cong (V, B). \]

In particular if $G$ is an orthogonal group, then even parts must occur with even multiplicity in $\lambda$, and if $G$ is symplectic, then odd parts must occur with even multiplicity in $\lambda$. 

\subsection{Generalised Whittaker models} 

We now define the generalised Whittaker representations $W_\gamma$ associated to the nilpotent orbit $\gamma$ and the associated generalised Whittaker models. These are also called generalised Gelfand-Graev representations and models (in analogy with the finite field case), but we use the name Whittaker in keeping with the rest of the paper.  

\subsubsection{Even nilpotent orbits}\label{evengamma} The most typical case we will consider is when $U=U^+$, that is, the nilpotent orbit under consideration is \textit{even}. 

\begin{definition}
In this case, we denote 
\begin{equation}
\label{defevenW}
W_{\gamma,\psi} := \ind^G_{M_\gamma U} \chi_\gamma
\end{equation}
with the trivial $M_\gamma$ action on $\chi_\gamma$. 
Denote also for $\pi\in \Irr(G)$ \[ W_{\gamma,\psi}(\pi):=\Hom_G (\ind^G_{M_\gamma U} \chi_\gamma, \pi^\vee) \cong \Hom_G (\pi, \Ind^G_{M_\gamma U} \chi_\gamma) \cong \Hom_{M_\gamma U} (\pi, \chi_\gamma)  \] 
This is called the space of generalised Whittaker functionals of $\pi$ (associated to $\gamma$). 
\end{definition}

For $\gamma$ a regular nilpotent orbit, it is easy to check that the character $\chi_\gamma$ of $U$ is generic (in the sense that its stabiliser is as small as possible, among all unitary characters of $U$). In fact $\chi_\gamma$ is generic for all even nilpotent orbits $\gamma$.

\subsubsection{Non-even nilpotent orbits}\label{notevengamma} If $\g_1\ne 0$, then we have an isomorphism
${\rm ad}(f)|_{\g_{1}}:\g_{1}\longrightarrow \g_{-1}$ coming from $\mfsl_2$-theory,
which allows us to transfer the non-degenerate pairing $\kappa$ between $\g_{-1}$ and $\g_1$ to a symplectic structure $\kappa_1$ on $\g_{1}$. Precisely, it is as follows:
\begin{equation}
\label{defsymg-1}
\kappa_{1}(v,w)=\kappa({\rm ad}(f)v,w)=\kappa(f,[v,w]), \qquad \mbox{for all  $v$, $w\in \g_{1}$}.
\end{equation}

Note that $\mf u/\mf u^+$ hence carries a $M_\gamma$-invariant symplectic form $\kappa_1$. 

Now consider $H_{\gamma}$ (not to be confused with the use of $H$ for a subgroup of $G$), the Heisenberg group associated to the symplectic $(\g_{1}, \kappa_{1})$. That is, $H_{\gamma}=\g_{1}\times F$, with $F$ central, and $(v,0)(w,0)=(v+w,\kappa _{1}(v,w)/2)$ for all $v$, $w \in \g_{1}$.

We have a group homomorphism $\alpha_{\gamma}:U\rightarrow H_{\gamma}$ given by
\[
\alpha_{\gamma}(\exp v \exp u)=(v,\kappa(f,u)), \qquad \mbox{for all $v\in \g_{1}$, $u\in \mf u^+$}. \label{eq:alphagammadefinition}
\]

Denoting $U' = \exp(\ker (\kappa_f|_{\mf u^+})) = \ker(\alpha_{\gamma}) $, we may think of this as exhibiting the structure of $U/U'$ as a Heisenberg group with center $U^+/U'$. 

If $\omega_\psi$ is the unique (by the Stone-von Neumann theorem) (smooth, irreducible, unitary) representation of $H_\gamma$ such that its center acts by $\psi$, then note that the action of $U^+/U'$ is given by \[
(\exp u)v=(0,\kappa(f,u))v=\psi(\kappa(f,u))v=\chi_{\gamma}(\exp u)v \qquad \mbox{for all  $u\in\mf u^+$,$v\in \omega_\psi$},
\]
i.e. it acts by the character $\chi_\gamma$. 

Since $M_\gamma$ preserves the symplectic form $\kappa_1$, in similar manner to how the Weil representation is constructed from the representation $\omega_\psi$ of the Heisenberg group, here, we have a representation on $\omega_\psi$, of some central cover of $M_\gamma$ which we denote as $\tilde{M_\gamma}$. It is the pre-image of $M_\gamma$ in ${\rm Mp}(\g_1)$. 

For a genuine representation $\rho$ of $\tilde{M_\gamma}$ (with $U$ acting trivially), $\rho\otimes \omega_\psi$ descends to an actual representation of $M_\gamma U$ (if we are working with the metaplectic cover as a double cover). 

\begin{definition}\label{def:GenWhittakerNoneven}
We denote
\begin{equation}
\label{defgeneralW}
W_{\gamma,\rho,\psi} := \ind^G_{M_\gamma U} \rho\otimes\omega_\psi
\end{equation}
and
\[ W_{\gamma,\rho,\psi}(\pi):=\Hom_G (\ind^G_{M_\gamma U} \rho\otimes\omega_\psi, \pi^\vee) \cong \Hom_G (\pi, \Ind^G_{M_\gamma U} \rho\otimes\omega_\psi) \cong \Hom_{M_\gamma U} (\pi, \rho\otimes\omega_\psi)  \] 

This is called the generalised Whittaker model of $\pi$ (associated to $\gamma$ and $\rho$). 

More generally, $\rho$ may be a (genuine) representation of $\tilde{H}$ for $H$ a reductive subgroup of $M_\gamma$, and we may form the corresponding $W_{\gamma,\rho,\psi}$ and $W_{\gamma,\rho,\psi}(\pi)$ with $H$ in place of $M_\gamma$. In most applications, $H$ will be a relatively big (e.g. finite index) subgroup of $M_\gamma$. 
\end{definition}

Note that in the case of even nilpotent orbits (Section \ref{evengamma}), we have $\tilde{M_\gamma}=M_\gamma$, $\omega_\psi=\chi_\gamma$, and we have essentially taken $\rho$ to be the trivial representation, so that we may henceforth use the same notation for both cases. 

\begin{rem}\label{rem:NoIsotypic}
    The preceding discussion suggests that one may omit the representation $\rho$ of $\tilde{M_\gamma}$ and instead consider the generalised Whittaker model as a representation of $G\times \tilde{M_\gamma}$, and indeed one commonly does so, for instance when dealing with the Bessel and Fourier-Jacobi models. See Remarks \ref{rem:NoIsotypic3} and \ref{rem:NoIsotypic2} for further discussion on this aspect. 
\end{rem}

Finally, since $\psi$ is fixed throughout this paper, where there is no danger of confusion we will sometimes drop the subscripts $_\psi$. 

\subsubsection{The choice of $\rho$}\label{sec:FourierJacobiInduction} Observe that in the even orbit case, there is a natural `canonical' choice of $\rho$: the trivial one. In the non-even case, one would ideally also like to have a `canonical' choice of $\rho$ and hence `canonical' choice of generalised Whittaker model. Conceptually, this should be achieved by choosing the `smallest'\footnote{in the sense of Gelfand-Kirillov dimension} possible $\rho$. 

It is probably not instructive to make such a choice explicit in full generality, since such a choice will depend, for instance, on whether the induced cover $\tilde{M_\gamma}$ is split or not. 

Therefore, in this subsection, we will record such a choice in one especially pertinent case to be considered in this paper (the `hook-type' case for symplectic groups, cf. Section \ref{sec:HookType}). More precisely, this is the case where $M_\gamma$ (or a subgroup $H\subseteq M_\gamma$) is precisely isomorphic to the symplectic group $\Sp(\g_1)$. 

In this case, $\rho$ will be taken to be the dual of the Weil representation $\omega_\psi^\vee$, of the metaplectic cover ${\rm Mp}(\g_1)$ (which will also allow us to work with the metaplectic cover as an $S^1$-cover), with again $U$ acting trivially.

One then has:

\begin{lem}\label{lem:FourierJacobiInduction}
As representations of $M_\gamma U$, \[ \omega_\psi^\vee \otimes \omega_\psi \cong \ind_{M_\gamma U^+}^{M_\gamma U} \chi_\gamma,\] with $M_\gamma$ acting trivially on $\chi_\gamma$. The isomorphism is given by formation of matrix coefficients: \[ s^\vee \otimes s \mapsto \big(mu \mapsto \langle s^\vee,  mu\cdot s \rangle\big)\] for all $s^\vee \in \omega_\psi^\vee, s \in \omega_\psi, mu \in M_\gamma U$.
\end{lem}
\begin{proof}
For the action of $M_\gamma$, one simply notes the standard fact \cite[Remark 2.11]{Pr} that  \[ \omega_\psi^\vee \otimes \omega_\psi \cong \mathscr{S}(\g_1) \] as representations of $\Sp(\g_1)$, where $\mathscr{S}(\g_1)$ denotes the space of (locally constant) functions on $\g_1$ with compact support, and $U/U^+\cong \mf{u}/\mf{u}^+ \cong \g_1$. 

Similarly for the action of $U$, this follows from \cite[Proposition 3.2.11]{Li}. 

It is straightforward to check that the respective isomorphisms for the $M_\gamma$- and $U$- actions are furnished by the same (natural) map of formation of matrix coefficients, as given above. 
\end{proof}

Therefore in this case we have: \begin{equation}
\label{defFJW}
W_{\gamma,\rho,\psi} \cong \ind^G_{M_\gamma U^+} \chi_\gamma
\end{equation}
from which the similarity to the even orbit case is immediately apparent. 

\section{Hamiltonian spaces and quantization}\label{sec:HamiltonianQuantization}

In the next two sections, we introduce the theory of duality of hyperspherical varieties as set out in \cite{BZSV}. We begin in this section by recalling the necessary preliminaries from symplectic geometry, as well as the classical philosophy of geometric quantization as a bridge between symplectic geometry and representation theory. 

\begin{rem}
We will be dealing with unitary representations in this section only, to illustrate the philosophy of quantization. For the rest of this paper, we will work in the analogous setting of smooth representations (hence all inductions are taken to be smooth, etc.). 
\end{rem}

\subsection{Hamiltonian spaces} We first review some preliminaries from symplectic geometry; one good reference is \cite[Chapter 1]{CG}.

\begin{definition} (Hamiltonian $G$-spaces)
A \textit{Hamiltonian $G$-space} (or $G$-variety) is a smooth, symplectic variety $M$ with a $G$-action (from the right, unless otherwise specified) by symplectomorphisms and a $G$-equivariant \textit{moment map} \[ \mu: M\rightarrow \g^\ast.\] 

\noindent The moment map $\mu$ must satisfy the following: \begin{itemize}
    \item Each $X\in \g$ induces a vector field $\rho(X)$ on $M$ by `differentiating' the $G$-action, which further induces a 1-form on $M$ by contracting with the symplectic form $\omega$: \[ Y \mapsto \omega(\rho(X), Y). \]
    On the other hand $X$ and $\mu$ also define a 1-form on $M$ via \[ d\big(m\mapsto (\mu(m))(X)\big).\]
    These two 1-forms must coincide. 
\end{itemize}

\end{definition}

\begin{definition}\label{def:PoissonBracket} (Poisson bracket)
Given two regular functions $f_1,f_2$ on $M$, we define the Poisson bracket $\{ f_1,f_2\}$ as follows: the two 1-forms $df_1,df_2$ are the contractions with $\omega$ of some (unique) vector fields $X_{f_1},X_{f_2}$ respectively. Then take \[ \{f_1,f_2\} := \omega(X_{f_1},X_{f_2}). \]
This makes the ring of regular functions on $M$ a Poisson algebra. 
\end{definition}

\begin{example}
Any symplectic vector space $(W, \langle -,-\rangle)$ is naturally a Hamiltonian $\Sp(W)$-space with the moment map \[ \mu : w \mapsto \big( X\mapsto \frac{1}{2}\langle Xw,w\rangle  \big) \quad \text{for $w\in W$ and $X\in\g$}.\]
\end{example}

\begin{example}
Any cotangent bundle $T^\ast X$ (for $X$ a $G$-variety) is naturally a symplectic variety with the symplectic form $\omega=d\lambda$, where $\lambda$ is the tautological 1-form pairing tangent and cotangent vectors. It is then naturally a Hamiltonian variety, with the moment map \[ \mu : p \mapsto \big( Y \mapsto -\lambda(\rho(Y))|_p \big) \quad \text{for $p\in T^\ast X$ and $Y\in\g$}. \]
\end{example}

\subsection{Quantization and examples} According to the classical philosophy of geometric quantization \cite{GuSt}, and as explained in \cite{Ga}, one may construct from each Hamiltonian $G$-space $M$ a (unitary) representation of $G$, which we call its quantization. 

\begin{example}\label{ex:QuantizationWeil} (Weil representation)
Consider a symplectic vector space $M:=W$ with $G:=\Sp(W)$ acting on it. Choosing a polarisation $W=X\oplus Y$ with $X,Y$ Lagrangians, the Weil representation, which can be realised on $L^2(Y)$, may be thought of as a quantization of the Hamiltonian $\Sp(W)$-space $W$. 
\end{example}

\begin{rem}\label{rem:Anomaly} (Anomaly)
Note that the Weil representation is not a representation of $\Sp(W)$ but of the metaplectic cover ${\rm Mp}(W)$. In the language of \cite{BZSV}, this is because $W$ has `anomaly' (which can be detected via Betti or étale cohomology), and anomalous varieties are at present excluded from the expectations of duality of hyperspherical varieties. None of the new examples of hyperspherical duality that we will exhibit in this paper will be anomalous. 

One should still work with the quantization as a representation of the metaplectic cover where appropriate, focusing on the cases where the quantization descends to an actual representation of an algebraic group, as in the remarks before Definition \ref{def:GenWhittakerNoneven}. 
\end{rem}

\begin{example}\label{ex:QuantizationCotangent} (Cotangent bundles)
The quantization of a cotangent bundle $T^\ast X$, for $X$ a $G$-variety, should be the space of functions $L^2(X)$, as a unitary representation of $G$.
\end{example}

\subsection{Symplectic reduction and induction}\label{sec:SymplecticOperations} This philosophy further postulates that many standard operations in symplectic geometry correspond to standard operations in representation theory. We shall review two of the most pertinent ones below: symplectic reduction and symplectic induction, which correspond respectively to formation of coinvariant spaces (or, more generally, multiplicity spaces) and induction of representations.  

\begin{definition}\label{def:SymplecticReduction} (Symplectic reduction)

The symplectic reduction of a Hamiltonian $G$-space $M$ is defined as \[ M\times^G_{\g^\ast} \{0\}.\]

The notation $\times^G_{\mf g^\ast}$ denotes the fiber product of $M$ and $\{0\}$ over $\mf g^\ast$ (via the moment maps), modulo the action of $G$ on $M$ (assuming the quotient exists as a scheme). 
    
\end{definition}

As explained in \cite{Ga}, quantization of symplectic reduction corresponds to taking $G$-coinvariant spaces. 

Furthermore, one may replace the trivial space $\{0\}$ with a coadjoint orbit $\mathscr{O} \subset \g^\ast$ corresponding under quantization to an irreducible representation $\rho$ of $G$. Quantization of the symplectic reduction then corresponds to taking the $\rho$-multiplicity space. 

Extending this idea further, given two Hamiltonian $G$-spaces $M_1$ and $M_2$, one may also consider the symplectic reduction of $M_1\times M_2^-$ (where $M_2^-$ is $M_2$ with its symplectic form and moment map negated), which corresponds under quantization to the formation of Hom spaces \cite{RZ}.

\begin{definition} (Symplectic induction)
    
We define the symplectic induction of a Hamiltonian $H$-space $S$ from $H$ to $G$ as 
\begin{equation} M:= S\times^H_{\mf h^\ast} T^\ast G\cong (S\times_{\mf h^\ast} \g^\ast)\times^H G\end{equation} 

Some remarks are in order:
\begin{itemize}
\item Here $H$ acts on $T^\ast G$ from the left, and we have the identification $T^\ast G \cong \g^\ast\times G$, where the moment map is the projection onto the first factor.

\item The notation $\times^H_{\mf h^\ast}$ denotes the fiber product of $S$ and $T^\ast G$ over $\mf h^\ast$ (via the moment maps), modulo the relation $(sh,x)\sim (s,hx)$ for $s\in S, h\in H, x\in T^\ast G$ (assuming the quotient exists as a scheme). 

\item The moment map for $M$ is induced by the right (coadjoint) action of $G$ on $\g^\ast$ sending $\phi \mapsto {\rm Ad}(g^{-1})\phi$. 
\end{itemize}
\end{definition}

The quantization of symplectic induction corresponds to induction of representations from $H$ to $G$. 

\begin{rem}\label{rem:FrobeniusReciprocitySymplectic}
With the operations of symplectic induction and reduction, one may readily formulate further analogues of other standard constructions in representation theory; for instance, symplectic analogues of Frobenius reciprocity have been first studied in \cite{GuSt3}. See \cite{RZ} (in particular Theorem 3.4) for a detailed discussion. 

One may further hope that this can be formalised in the sense of forming a category of Hamiltonian spaces, and this is indeed possible in the setting of \textit{shifted symplectic geometry} \cite{PTVV}, which involves the machinery of derived geometry. For our purposes, the classical constructions of symplectic induction and reduction are sufficient. 
\end{rem}

\subsection{} We now examine some important examples of quantizations of cotangent bundles to illustrate symplectic reduction and induction. 

\begin{example}\label{ex:QuantizationWhittaker} (Cotangent bundles)

Suppose $H\subset G$ are groups, then the cotangent bundle $M:=T^\ast(H\bs G)$ may be identified with the symplectic induction of the trivial $H$-space $\{0\}$ from $H$ to $G$, which is \[ \{0\} \times^H_{\mf h^\ast} T^\ast G. \]  Its quantization should be the (unitary) induction $L^2(H \bs G) = (L^2-)\Ind^{G}_H \C $. Most often, we consider the case where $H\bs G$ is a spherical variety. 

The cotangent bundle $T^\ast(H\bs G)$ may also be thought of as the symplectic reduction of $T^\ast G$ with respect to $H$ (with $H$ now acting from the right via $g\mapsto h^{-1}g$), which corresponds to taking the $H$-coinvariant space of the regular representation $L^2(G)$ of $G$.
\end{example}

\begin{example}\label{ex:TwistedCotangent} (Twisted cotangent bundles)

Suppose now $H=N$ is a unipotent subgroup of $G$.

By shifting the moment map of the trivial $N$-space $\{0 \}$, we obtain twisted cotangent bundles \[ M:=(\lambda+\mf n^\perp)\times^N G \rightarrow N\bs G\] which quantizes to \[L^2(N,\psi \bs G) = (L^2-)\Ind^{G}_N \psi = \{ f: G\rightarrow \C \mid f(ng)=\psi(n)f(g)\} \] (the choice of $\lambda$ corresponds to choice of $\psi$). In general, a representation induced from a character can be thought of as the quantization of some twisted cotangent bundle. 

This example includes the usual Whittaker case, when $N$ is a maximal unipotent subgroup of $G$. 

\end{example}

\section{Relative Langlands duality}\label{sec:RelativeLanglandsDuality}

In this section, we introduce the theory of duality of hyperspherical varieties as set out in \cite{BZSV}.

\subsection{Hyperspherical varieties} The central objects of study in \cite{BZSV} are a class of Hamiltonian $G$-varieties defined over $\C$ (or an algebraically closed field of characteristic zero), called \textit{hyperspherical} varieties. 

\begin{definition}\label{def:Hyperspherical} (Hyperspherical varieties) 
A \textit{hyperspherical variety} is a smooth Hamiltonian $G$-variety equipped with a grading (that is, a commuting $\G_m$-action), such that:

\begin{itemize}
    \item it is affine;
    \item it satisfies the \textit{multiplicity-free, or coisotropic}, condition: the ring of $G$-invariant functions on $M$ is Poisson-commutative (cf. Definition \ref{def:PoissonBracket}). 
\end{itemize}

One also requires $M$ to satisfy several technical conditions: its generic stabiliser is connected, its moment map image meets the nilcone, and the $\G_m$-action is ``neutral" (which we will not define here). 

However, the most important condition is the multiplicity-free condition, and in what follows we will often (without loss of generality) ignore the other technical conditions, cf. the remarks after Theorem \ref{MainStructureThm}.  

\end{definition}

Under the philosophy of quantization as explained in Section \ref{sec:HamiltonianQuantization}, the multiplicity-free condition corresponds to the multiplicity-free property for representations \cite{GuSt2}.

\subsection{Whittaker induction}\label{sec:WhittakerInduction}

Continue the notation of Section \ref{sec:NilpOrbits}. 

\begin{definition}\label{def:WhittakerInduction} (Whittaker induction)
Consider any reductive subgroup $H$ of $G$ and a commuting $\SL_2$ (giving rise to a homomorphism $H\times \SL_2\rightarrow G$). Let $S$ be a symplectic $H$-vector space (or more generally a Hamiltonian $H$-space, but we do not need this). The Whittaker induction of $S$ along $H\times \SL_2\rightarrow G$ is defined as follows:

Let $\gamma$ be the $\mfsl_2$ triple corresponding to the $\SL_2$ factor; we have seen that $\mf u/\mf u^+$ carries a $M_\gamma$-invariant symplectic form $\kappa_1$, and hence can be naturally considered as a Hamiltonian $H$-space (since $H$ centralises $\gamma$) via the adjoint action of $H$. 

We in fact consider it as a Hamiltonian $HU$-space where $U$ acts additively via the identification $U/U^+\cong \mf{u}/\mf{u}^+$, and the moment map $\mu_U:\mf u/\mf u^+\rightarrow \mf u^\ast$ is shifted by $\kappa_f$, that is, $\mu_U(u) = \kappa_1(u)+\kappa_f$, where $\kappa_1:\mf u/\mf u^+\rightarrow (\mf u/\mf u^+)^\ast$ is the identification via the symplectic form. 

Then the Whittaker induction of $S$ is defined to be the symplectic induction of $S\times (\mf u/\mf u^+)$ from $HU$ to $G$: 
\begin{equation}
\label{defWhittakerInduction} 
(S\times (\mf u/\mf u^+))\times^{HU}_{(\mf h+\mf u)^\ast} T^\ast G \cong ((S\times (\mf u/\mf u^+))\times_{(\mf h+\mf u)^\ast} \g^\ast)\times^{HU} G
\end{equation}
\end{definition}

Comparing Section \ref{notevengamma} and Definition \ref{def:WhittakerInduction}, we see that since $H$ is a subgroup of the centraliser $M_\gamma$ of the $\SL_2$ factor, then under the philosophy of quantization, Whittaker induction corresponds precisely to the formation of the generalised Whittaker representations $W_{\gamma,\rho,\psi}$ as in Section \ref{notevengamma}: 

\begin{itemize}

    \item $S$ corresponds to $\rho$;
    \item $\mf u/\mf u^+$ corresponds to the oscillator representation $\omega_\psi$ of $U$ with the associated (Weil) representation of $H$;
    \item the symplectic induction from $HU$ to $G$ corresponds to the induction of representations from $HU$ to $G$. 
\end{itemize}

It is this case that we will be working with throughout this paper. In particular, the choice of $S$ should correspond to the canonical choice of $\rho$ as described in Section \ref{sec:FourierJacobiInduction}. 

\begin{rem}\label{rem:Grading} (Grading)
When $S$ has a grading (i.e. commuting $\G_m$-action), the Whittaker induction of $S$ can also be given a natural grading. However, since we do not make essential use of the grading in this paper, we omit the details (which are rather lengthy and technical). It will suffice to mention that every symplectic $H$-vector space $S$ is naturally graded via linear scaling, and so every Whittaker-induced space from a symplectic vector space also carries a corresponding natural grading. 
\end{rem}

\subsubsection{Simplifying the Whittaker induction} The definition (\ref{defWhittakerInduction}) of Whittaker induction, while corresponding nicely under quantization to the formation of generalised Whittaker representations, is geometrically unwieldy. It is possible, via the theory of Slodowy slices, to simplify the Whittaker induction somewhat.

In particular, \cite[Lemma 2.1]{GaGi} states that one has an isomorphism \begin{equation}\label{SlodowySlice} U\times (f+\g^e)\rightarrow f+\mf u^{+,\perp}\end{equation} given by the action map of $U$ on $(f+\g^e)$, where $\g^e$ is the centraliser of $e$ (considered as a subspace of $\g^\ast$ via $\kappa$). 

Now note that \begin{equation}\label{FiberProd}
    (S\times (\mf u/\mf u^+))\times_{(\mf h+\mf u)^\ast} \g^\ast
\end{equation}  
may be identified with the set of pairs $(s,x)$ for $s\in S$ and $x\in \g^\ast$, such that 
\begin{itemize}
    \item the restrictions of $\mu(s)$ and $x$ to $\h$ are equal ($\mu$ is the moment map for $S$), and
    \item that $x$ restricts to $f$ on $\mf u^+$, that is, $x\in f+\mf u^{+,\perp}$ (noting that we have used the $f$- or $\kappa_f$-shifted moment map for $(\mf u/\mf u^+)$),
\end{itemize}
since then the corresponding element of $(\mf u/\mf u^+)$ is uniquely determined by $(s,x)$. 

Combining (\ref{SlodowySlice}) and (\ref{FiberProd}), one therefore sees that we have a ($H$-equivariant) isomorphism \[ (S\times (\mf u/\mf u^+))\times_{(\mf h+\mf u)^\ast} \g^\ast \cong (S\times_{\mf h^\ast} (f+\g^e) )\times U,  \]

Hence the Whittaker induction can be written as 
\begin{equation}
\label{defWhittakerInduction2} 
(S\times_{\mf h^\ast} (f+\g^e) )\times^{H} G
\end{equation}

We remark that when $S$ is trivial, we have $S\times_{\mf h^\ast} (f+\g^e) = \mf h^\perp \cap (f+\g^e)$. 

Such a rewriting as in (\ref{defWhittakerInduction2}) makes clear the geometric meaning of the Whittaker induction $M$: it is always an (affine) bundle over $H\bs G$. In other words, it is the subgroup $H$ which plays the biggest role in the geometry of $M$. 

Also, since we will later work with the full orthogonal groups (rather than the special orthogonal groups), (\ref{defWhittakerInduction2}) also makes clear that one can formulate the analogous hyperspherical dual pairs for special orthogonal groups without much issue.

\subsection{Structure theorem} The main structure theorem proved in \cite{BZSV} is as follows:

\begin{thm}\label{MainStructureThm}
Suppose $M$ is a hyperspherical $G$-variety. Then there is a reductive subgroup $H$ of $G$ and a commuting $\SL_2$ (giving rise to a homomorphism $H\times \SL_2\rightarrow G$), and a symplectic $H$-vector space $S$, such that $M$ is the \textit{Whittaker induction} of $S$ along $H\times \SL_2\rightarrow G$. 
\end{thm}

Conversely, to check that the Whittaker induction of $S$ along $H\times \SL_2\rightarrow G$ is hyperspherical, it suffices to check the multiplicity-free condition (and that the generic stabiliser is connected). We have seen from (\ref{defWhittakerInduction2}) that the Whittaker induction is automatically affine, and it is shown in \cite{BZSV} that the other technical conditions of Definition \ref{def:Hyperspherical} are also satisfied. 

\begin{rem} (On connectedness)
    Note that in \cite{BZSV} the group $G$ is usually assumed connected; for us, we do not require that $G$ be connected, since we will work with the orthogonal group rather than the special orthogonal group. Consequently, the condition that the generic stabiliser be connected will have correspondingly less significance for us. 

    However, since this condition is of importance in certain contexts and is a condition that needs to be independently checked, let us remark that it will be possible, using (\ref{defWhittakerInduction2}) and the results of \cite{GL} for instance, to verify that the generic stabilisers are connected in the specific cases that we will consider in this paper, after replacing the groups involved with their identity components. While we have not verified this in every case, we do not have any reason to expect otherwise.

\end{rem}

\begin{rem}\label{rem:Rationality} (Rationality)
    With the structure theorem in hand, now one defines (forms of) hyperspherical varieties over non-algebraically closed fields (the obvious one being our local field $F$), via the algebraic datum \[ H\times \SL_2\rightarrow G, \quad H\rightarrow \Sp_{2n}\] over $F$. (To be more precise, $H,G$ are reductive group schemes over $F$ and $H\rightarrow G$ is a closed immersion.) 
    
    The expectation being that, for each $M$ (over $\C$) there will be a distinguished `split form' of $M$ defined over arithmetic fields $F$ (and there is an expected construction of such a form in the untwisted case). 

    For our purposes in working with generalised Whittaker models, this is the most reasonable and natural point of view, and resolves potential issues to do with rationality. 

\end{rem}

\subsection{Duality}\label{sec:DualityExpectations} The key point of the theory is that there is expected to exist a \textit{duality} \[G\circlearrowright M\longleftrightarrow M^\vee \circlearrowleft G^\vee\] of (anomaly-free, cf. Remark \ref{rem:Anomaly}) hyperspherical varieties (over $\C$). Each such hyperspherical dual pair $(M,M^\vee)$ encodes an instance of the relative Langlands program with corresponding (conjectural) statements at the local and global levels. 

However, there is at present no definition of the duality $M\leftrightarrow M^\vee$ (which can be proved to satisfy the desiderata of the duality theory). Hence, for a given pair $(M,M^\vee)$, the expected duality of $M$ and $M^\vee$ should be verified via the conjectures it entails. Here we focus on a (smooth) local incarnation or counterpart of the conjectures. 

\begin{claim}\label{ex:relLanglandsQuantizationSmooth}
The irreducible representations of $G$ of Arthur type which occur as quotients of a quantization of $M$ belong to Arthur packets whose Arthur parameters factor through the morphism defining the hyperspherical variety $M^\vee$: \[ X^\vee(\C) \times \SL_2(\C)\rightarrow G^\vee(\C). \] 

\end{claim}

\begin{rem}\label{rem:DualExpectation}
    Because the theory is a \textit{duality}, one in fact expects a pair of statements arising from this dual pair $M,M^\vee$ by exchanging the roles of $M$ and $M^\vee$, often relating two \textit{a priori} unrelated (branching) problems:

\begin{quote}
    The irreducible representations of $G^\vee$ of Arthur type which occur as quotients of a quantization of $M^\vee$ belong to Arthur packets whose Arthur parameters factor through the morphism defining the hyperspherical variety $M$: \[ H(\C) \times \SL_2(\C)\rightarrow G(\C). \] 
\end{quote}

    We may refer to this as the `dual problem', and this is the key feature of the duality theory. 
\end{rem}

The conjecture hence characterizes the spectral decomposition of the quantization of $M$ as the image of a certain Langlands functorial lifting.

In the smooth setting, it is known that there are additional subtleties and such a formulation in terms of Arthur parameters is only true as a guiding principle, or up to an approximation. Nonetheless, one expects a natural map realising a lifting of irreducible representations, and in this paper, such a lifting will be facilitated by the theta correspondence (Section \ref{sec:ThetaCorrespondence}). 

It is an interesting future problem to investigate the precise statements in the local $L^2$-setting and even the global setting. This will involve at a minimum the computation of the relevant local relative characters in the unramified setting (as was first done in \cite{Sa1} and later also in \cite{WZ}).

\section{Hyperspherical Whittaker models}\label{sec:HypersphericalClassification}

In this section, we determine an upper bound for the possible generalised Whittaker models for the orthogonal and symplectic groups which arise from hyperspherical varieties (over $\C$), cf. Section \ref{sec:WhittakerInduction}. This provides an effective upper bound for the possible generalised Whittaker models which may be contained in the conjectural class admitting a duality theory. 

\subsection{} To do so we first need an effective criterion for hypersphericality. Continuing the notation of Section \ref{sec:RelativeLanglandsDuality}, we record: 

\begin{prop}\label{prop:HypersphericalCriterion}
If $M$ is hyperspherical, then $H\bs L$ is a (smooth, affine) spherical $L$-variety, where $L$ is the Levi factor of the parabolic $P=LU$ associated to the $\mfsl_2$-triple $\gamma$. In particular, $H$ is a spherical subgroup of $M_\gamma$, and $M_\gamma$ is a spherical subgroup of $L$. 
\end{prop}
\begin{proof}
    It is shown in \cite{BZSV} that if $M$ is coisotropic, then $HU\bs G$ is spherical. (We remark that it is known from the theory of spherical varieties that a twisted cotangent bundle $M=T^\ast(X,\psi)$ is coisotropic if and only if $X$ is a spherical variety. The proof in \cite{BZSV} essentially extends this to general Whittaker-induced $M$.) 

Now again from the theory of spherical varieties (cf. the theory of Whittaker-type induction \cite{SV}), we may view $HU\bs G$ as the parabolic induction from $H\bs L$ along the parabolic $P=LU$, with then $H\bs L$ a (smooth, affine) spherical $L$-variety. 

(A direct proof of the key spherical property is as follows: Let $B_L,B$ be Borel subgroups of $L,G$ with $B_L=B\cap L$ and $P\supseteq B\supseteq U$; we have the natural embedding \[ B_L\bs L\hookrightarrow B\bs G. \] Now every spherical variety has only finitely many orbits under a Borel subgroup \cite{Kn}, so $B\bs G$, and hence too $B_L\bs L$, has finitely many $HU$-orbits, and hence finitely many $H$-orbits. One of these orbits must hence be dense, which means $H\bs L$ is spherical as desired.)
\end{proof}

\begin{rem}
It is in fact shown in \cite{BZSV} that $M$ is coisotropic if and only if $HU\bs G$ is spherical and $(S\times (\mf u / \mf u^+))$ is coisotropic for the generic stabiliser of $T^\ast(HU\bs G)$ (in particular for $H$). This latter condition can be checked for each given case using the tables of \cite{Kn2}. Therefore, modulo the choice of $S$ and checking that the generic stabiliser of $M$ be connected (which in practice also corresponds roughly to the exclusion of type N spherical roots \cite{BZSV}), the upper bound we obtain in this section is relatively close to a precise classification of nilpotent orbits that give rise to hyperspherical varieties. 
\end{rem}

In view of the above proposition, and since we are interested in determining the possible $\gamma$ which gives rise to hyperspherical varieties, let us without loss of generality assume for the rest of this section that $H=M_\gamma$ is the centraliser of the corresponding $\mfsl_2$ triple $\gamma$. 

Further, for simplicity, let us first consider the case where the corresponding nilpotent orbit is even.

\begin{rem}
We briefly consider the case when the nilpotent orbit is not even at the end of this section, in Proposition \ref{prop:Noneven}. Given the relative complexity of the characterization we obtain, it is likely that a cleaner classification in generality can only be achieved after a fuller theory of combinatorial datum for general hyperspherical varieties is developed (in line with that for spherical varieties), which is one of the main open questions arising from \cite{BZSV}. 
\end{rem}

In what follows we will also make use of the following dimensional consideration:

\begin{lem}
If $H\bs L$ is spherical, then $\dim L - \dim H \le \dim B$, where $B$ is a Borel subgroup of $L$.
\end{lem}

\begin{proof}
This is an immediate consequence of the fact that $B$ has an open dense orbit on $H\bs L$. 
\end{proof}

\subsection{Orthogonal groups} Suppose now $G$ is the orthogonal group $\O_n$, acting on an $n$-dimensional vector space $V$ equipped with an orthogonal form $B$.

Recall from Section \ref{NilpPartition} that the nilpotent orbits in $G$ are parameterised by the datum of the partition $\lambda=[l^{a_l},\dots, 1^{a_1}]$ of $n$, and the forms on the multiplicity spaces $(V_j,B_j)$, where even parts must occur with even multiplicity in $\lambda$. For even nilpotent orbits, all parts of the partition $\lambda$ have the same parity \cite{CM}. We have seen also that \[ H=M_\gamma \cong \prod_{j=1}^l G(V_j,B_j). \]

Let \[\lambda^t=[a_l+\dots+a_1,\dots, a_l]=[c_1,\dots,c_m]=[h^{b_h},\dots,1^{b_1}]\] denote the transpose partition of $\lambda$. 

If $\lambda$ has all even parts, then 
\[ H\cong \Sp_{a_l} \times \dots \Sp_{a_2}, \]
\[L\cong \GL_{h}^{\times b_h/2} \times \dots \times \GL_{1}^{\times b_1/2}.\]

If $\lambda$ has all odd parts, then 
\[ H\cong \O_{a_l} \times \dots \O_{a_1}, \]
\[L\cong \O_h \times \GL_{h}^{\times (b_h-1)/2} \times \dots \times \GL_{1}^{\times b_1/2}. \]

In other words, representing $\lambda$ by a Young tableaux $d$ (in the usual way), the factors of $H$ correspond to groups of rows of the tableaux with same length, while the factors of $L$ correspond to (pairs of) columns of the tableaux with same length. See Figure \ref{fig:ExampleHL} for an example. 

\begin{figure}[!h]
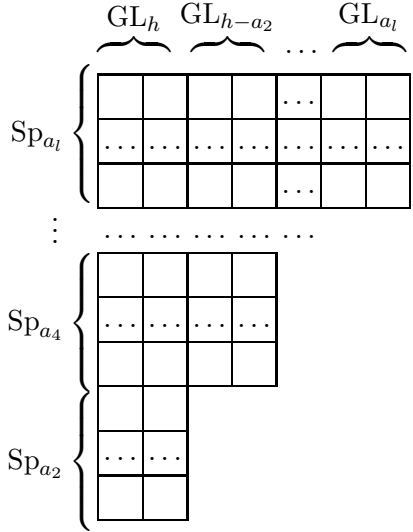

    \begin{center}
    
   \ytableausetup {mathmode,boxframe=normal,boxsize=1.5em,centertableaux} 
   \begin{tabular}{r@{}l}
   & $\overbrace{\hspace{2.5em}}^{\displaystyle \GL_h}$ $\overbrace{\hspace{2.5em}}^{\displaystyle \GL_{h-a_2}}$ $\dots$ $\overbrace{\hspace{2.5em}}^{\displaystyle \GL_{a_l}}$ \\
   \begin{tabular}{r@{}l} $\Sp_{a_l}\left\{\vphantom{\begin{ytableau}
    \quad \\ \quad  \\ \quad
 \end{ytableau}}\right.$ \\ $\vdots$ \hspace{1em} \\
 $\Sp_{a_4}\left\{\vphantom{\begin{ytableau}
    \quad \\ \quad  \\ \quad
 \end{ytableau}}\right.$  \\
 $\Sp_{a_2}\left\{\vphantom{\begin{ytableau}
    \quad \\ \quad  \\ \quad
 \end{ytableau}}\right.$  \end{tabular} & \begin{ytableau} \quad & \quad &\quad & \quad &\dots&\quad &\quad \\ \dots & \dots &\dots& \dots &\dots &\dots&\dots \\ \quad & \quad &\quad & \quad  &\dots &\quad&\quad \\ \none[\dots] & \none[\dots]& \none[\dots]& \none[\dots]& \none[\dots] \\ \quad & \quad &\quad & \quad \\ \dots &\dots &\dots &\dots \\ \quad &\quad &\quad &\quad \\ \quad & \quad \\ \dots &\dots \\ \quad &\quad   \end{ytableau} \\ \\
   \end{tabular} 

\end{center}

\caption{An illustration of the factors of $H$ and $L$ in one example where $\lambda$ has even parts.  }\label{fig:ExampleHL}
\end{figure}

Furthermore, $H$ is embedded in $L$ such that, for a factor $G(V_j,B_j)$ of $H$ and a factor $\O_i$ or $\GL_i$ of $L$, the composite projection $G(V_j,B_j) \hookrightarrow H \hookrightarrow L \twoheadrightarrow \O_i$ or $\GL_i$ is non-zero if and only if the corresponding rows and columns in $d$ share common cells. 

Now we may use the classification of smooth affine spherical varieties in \cite{KVS} to characterize the allowable nilpotent orbits. 

\begin{thm}\label{thm:OrthogonalClassification}
Let $G$ be the orthogonal group $\O_n$ and $M$ be a hyperspherical $G$-variety. By Theorem \ref{MainStructureThm}, it is obtained as the Whittaker induction along a map $H\times \SL_2\rightarrow G$. Let $\gamma$ be the nilpotent orbit determined by the $\SL_2$ factor. If $\gamma$ is even, then it corresponds to a partition $\lambda$ of the form 
\begin{itemize}
    \item $[2^{a_2}]$ (Shalika)
    \item $[n-a_1,1^{a_1}]$ (hook-type, corresponding to Bessel models \cite{GGP})
    \item finitely many low-rank exceptions \[[3,3], [4,4],[6,6].\] 
\end{itemize} 
\end{thm}
\begin{proof}
The classification of smooth affine spherical varieties $H'\bs G'$ is given in Tables 4 and 5 of \cite{KVS}. There, all the possible (indecomposable) pairs $\mf h'=\mf h'_1\oplus\dots\oplus \mf h'_s$, $\g' = \g'_1\oplus\dots\oplus \g'_r$ are listed, together with the combinatorial datum of a graph $\Gamma$ with vertices corresponding to the factors $\mf h'_1,\dots, \mf h'_s$ and $\g'_1,\dots, \g'_r$, and edges indicating if the composite projection $\mf h'_j \hookrightarrow \mf h' \hookrightarrow \g' \twoheadrightarrow \g'_i$ is non-zero. Here we always take $\mf h', \g'$ to be the commutator subalgebras of the Lie algebras of $H',G'$. In our case, we take $H'=H$ and $G'=L$, in view of Proposition \ref{prop:HypersphericalCriterion}. 

First, if $\mf h$ is trivial, then $\mf l$ must be trivial or some copies of $\mfsl_2$. In the former case, it is easy to see that one can only have the regular nilpotent orbit. In the latter case, a simple dimensional consideration shows that we can have at most one copy of $\mfsl_2$, and then that we can only have the partitions $[n-2,1,1]$, $[2,2]$, or $[3,3]$.

Therefore assume henceforth that $H$ has a factor $H_j=G(V_j,B_j)$ with non-trivial $\mf h_j$, corresponding to $a_j$ rows of $d$ with the same length $j$. 

(a) If all parts of $\lambda$ are even. Recall that even parts of $\lambda$ occur with even multiplicity. Then if $j\ge 4$, we must have in the graph $\Gamma$ a vertex corresponding to an $\mf{sp}(a_j)$ factor, connected to two or more vertices each corresponding to $\mf{sl}(a_j+r_1),\mf{sl}(a_j+r_2)$ factors for some even $r_1,r_2\ge 0$. An inspection of Table 5 of \cite{KVS} shows that this is only possible when $a_j=2$ and $r_1=r_2=0$ (that is, there are no other rows in $d$), and then only when $j\le 6$. So there are only finitely many low-rank exceptions here, namely $[4,4]$ and $[6,6]$. 

Hence $j=2$. Now if there are $r\ge 1$ rows of length $>2$ in $d$, then we must have $r\ge 2$ (because $r$ is even). Then one must have in $\Gamma$ a vertex corresponding to an $\mf{sp}(a_2)$ factor, connected to a vertex corresponding to an $\mf{sl}(a_2+r)$ factor with $r\ge 2$. An inspection of Tables 4 and 5 of \cite{KVS} shows that there is only one possibility, corresponding to the partition $[4,4,2,2]$. However, in this case it is not possible for the $H \bs L$ to be spherical, by a dimensional consideration.   

Otherwise, we are left with the partition $\lambda=[2^{a_2}]$ (which in fact corresponds to the Shalika case). 

(b) If all parts of $\lambda$ are odd. Note then $a_j\ge 3$. If $j\ge 3$, we must have in the graph $\Gamma$ a vertex corresponding to an $\mf{so}(a_j)$ factor, connected to two or more vertices, one corresponding to an $\mf{so}(a_j+r)$ factor and one corresponding to $\mf{sl}(a_j+r')$ for $r\ge r' \ge 0$. An inspection of Table 5 of \cite{KVS} shows that this is not possible. (Note that the partition $[3,3,3]$ appears to be a possible low-rank exception, but is not, due to the difference in embedding between $\mf{so}(3)\hookrightarrow \mfsl(3)$ and $\mfsl(2)\hookrightarrow \mfsl(3)$.)

Hence $j=1$. Now if there are $r\ge 2$ other rows of length $>1$ in $d$, then one must have in $\Gamma$ a vertex corresponding to an $\mf{so}(a_1)$ factor, connected to a vertex corresponding to an $\mf{ so}(a_1+r)$ factor with $r\ge 2$. When $r=2$ the tables in \cite{KVS} do not preclude this but it is readily checked that the $H\bs L$ cannot be spherical (because $\O(n)\bs \O(n+2)$ is not spherical for $n\ge 1$), so in fact $r\ge 3$. An inspection of Tables 4 and 5 of \cite{KVS} shows that this is not possible. 

So there is at most 1 other row of length $>1$ in $d$, and the corresponding partitions are $[1^{a_1}]$ (trivial partition) or $[n-a_1,1^{a_1}]$ (hook-type).
\end{proof}

\subsection{Symplectic groups} Suppose now $G$ is the symplectic group $\Sp_{2n}$, acting on a $2n$-dimensional vector space $V$ equipped with an symplectic form $B$.

Again from Section \ref{NilpPartition}, the nilpotent orbits in $G$ are parameterised by the datum of the partition $\lambda=[l^{a_l},\dots, 1^{a_1}]$ of $n$, and the forms on the multiplicity spaces $(V_j,B_j)$, where odd parts must occur with even multiplicity in $\lambda$. For even nilpotent orbits, all parts of the partition $\lambda$ have the same parity. We have seen also that \[ H=M_\gamma \cong \prod_{j=1}^l G(V_j,B_j). \]

Let \[\lambda^t=[a_l+\dots+a_1,\dots, a_l]=[c_1,\dots,c_m]=[h^{b_h},\dots,1^{b_1}]\] denote the transpose partition of $\lambda$. 

If $\lambda$ has all even parts, then 
\[ H\cong \O_{a_l} \times \dots \O_{a_2}, \]
\[L\cong \GL_{h}^{\times b_h/2} \times \dots \times \GL_{1}^{\times b_1/2}.\]

If $\lambda$ has all odd parts, then 
\[ H\cong \Sp_{a_l} \times \dots \Sp_{a_1}, \]
\[L\cong \Sp_h \times \GL_{h}^{\times (b_h-1)/2} \times \dots \times \GL_{1}^{\times b_1/2}.\] 

In other words, representing $\lambda$ by a Young tableaux $d$ (in the usual way), the factors of $H$ correspond to groups of rows of the tableaux with same length, while the factors of $L$ correspond to (pairs of) columns of the tableaux with same length.

Furthermore, $H$ is embedded in $L$ such that, for a factor $G(V_j,B_j)$ of $H$ and a factor $\Sp_i$ or $\GL_i$ of $L$, the composite projection $G(V_j,B_j) \hookrightarrow H \hookrightarrow L \twoheadrightarrow \Sp_i$ or $\GL_i$ is non-zero if and only if the corresponding rows and columns in $d$ share common cells. 

Now we have the analogous result to Theorem \ref{thm:OrthogonalClassification}. 

\begin{thm}\label{thm:SymplecticClassification}
Let $G$ be the symplectic group $\Sp_{2n}$ and $M$ be a hyperspherical $G$-variety. By Theorem \ref{MainStructureThm}, it is obtained as the Whittaker induction along a map $H\times \SL_2\rightarrow G$. Let $\gamma$ be the nilpotent orbit determined by the $\SL_2$ factor. If $\gamma$ is even, then it corresponds to a partition $\lambda$ of the form 
\begin{itemize}
    \item $[2^{a_2}]$ (Shalika)
    \item the `exceptional' partitions \[ [3,3,1^{2a}], [5,5,1^{2a}] \] for $a\ge 0$ 
    \item $[1^{2n}]$ (trivial orbit) or $[2n]$ (regular orbit). 
\end{itemize} 
\end{thm}

\begin{rem}
The hook-type partitions are not included for the symplectic group as they correspond to non-even nilpotent orbits; they correspond to Fourier-Jacobi models \cite{GGP} and will also be studied later in Section \ref{sec:HookType}, cf. Remark \ref{rem:SymplecticHookType}. 
\end{rem}

\begin{proof}
Keep the notation of the proof of Theorem \ref{thm:OrthogonalClassification}. 

First, if $\mf h$ is trivial, then $\mf l$ must be trivial or some copies of $\mfsl_2$. In the former case, it is easy to see that one can only have the regular nilpotent orbit. In the latter case, a simple dimensional consideration shows that we can have at most one copy of $\mfsl_2$, and then that we can only have the partition $[2,2]$.

Therefore assume henceforth that $H$ has a factor $H_j=G(V_j,B_j)$ with non-trivial $\mf h_j$, corresponding to $a_j$ rows of $d$ with the same length $j$. 


(a) If all parts of $\lambda$ are even. Note then $a_j\ge 3$. If $j\ge 4$, we must have in the graph $\Gamma$ a vertex corresponding to an $\mf{so}(a_j)$ factor, connected to two or more vertices each corresponding to $\mf{sl}(a_j+r_1),\mf{sl}(a_j+r_2)$ factors for $r_1,r_2\ge 0$. An inspection of Table 5 of \cite{KVS} shows that this is not possible. 

Hence $j=2$. Now if there are $r\ge 2$ other rows of length $>1$ in $d$, then one must have in $\Gamma$ a vertex corresponding to an $\mf{so}(a_1)$ factor, connected to a vertex corresponding to an $\mf{ sl}(a_1+r)$ factor with $r\ge 1$. An inspection of Tables 4 and 5 of \cite{KVS} shows that this is not possible. 

So we are left with the partition $\lambda=[2^{a_2}]$ (which in fact corresponds to the Shalika case). 

(b) If all parts of $\lambda$ are odd. Note that then all parts of $\lambda^t$ are even. It follows that all factors of $\mf h$ and $\mf l$ are non-trivial (they are all $\mf{sp}$ and $\mf{sl}$ factors). 

Considering the possible graph structures of $\Gamma$ and examining Table 5 of \cite{KVS}, then, one sees that there are only the following possible partitions: \[ [3,3], [5,5], [3,3,1^{2a}], [5,5,1^{2a}]. \]

\end{proof}

\subsection{Non-even orbits} We end this section by considering the case of non-even nilpotent orbits. Continuing the hypothesis of Theorems \ref{thm:OrthogonalClassification} and \ref{thm:SymplecticClassification}, we have:

\begin{prop}\label{prop:Noneven} (Non-even nilpotent orbits)
Let $G$ be either the orthogonal or the symplectic group, and $M$ be a hyperspherical $G$-variety. By Theorem \ref{MainStructureThm}, it is obtained as the Whittaker induction along a map $H\times \SL_2\rightarrow G$. Let $\gamma$ be the nilpotent orbit determined by the $\SL_2$ factor. Suppose $\gamma$ is not necessarily even, and has corresponding partition $\lambda$. 

Let $\lambda_{even}$ (resp. $\lambda_{odd}$) be the partitions formed by taking only the even (resp. odd) parts of $\lambda$. 

Then $\lambda_{even}$ and $\lambda_{odd}$ must each be of a form listed in Theorems \ref{thm:OrthogonalClassification} or \ref{thm:SymplecticClassification} respectively (according as $G$ is orthogonal or symplectic). 
\end{prop}

\begin{proof}
Note that $H\bs L$ decomposes as a direct product $H_{even}\bs L_{even} \times H_{odd} \bs L_{odd}$, corresponding to the even and odd parts respectively of $\lambda$, and that $H\bs L$ is spherical $\iff$ $H_{even}\bs L_{even}$ and $H_{odd} \bs L_{odd}$ are both spherical. 

It follows that the even and odd parts of $\lambda$ must be of the forms listed in Theorems \ref{thm:OrthogonalClassification} or \ref{thm:SymplecticClassification}.

\end{proof}

\begin{example} (Hook-type for symplectic groups)
For instance, the hook-type partitions $\lambda=[2a,1^{2b}]$ for the symplectic groups has even part $\lambda_{even}=[2a]$ and odd parts $\lambda_{odd}=[1^{2b}]$, both of which are listed in Theorem \ref{thm:SymplecticClassification}. 

\end{example}

We will see several more examples in Section \ref{sec:ExceptionalPartitions}, cf. Expectation \ref{ex:Exceptional}. 

\subsection{The $G\times M_\gamma$ case}
\begin{rem}\label{rem:NoIsotypic3} In \cite{FU}, a similar classification of nilpotent orbits is given in the case of hyperspherical $G\times M_\gamma$-varieties; see Remarks \ref{rem:NoIsotypic} and \ref{rem:NoIsotypic2}. That is, they essentially work with the datum \[  M_\gamma \times \SL_2 \rightarrow G\times M_\gamma \quad \text{and} \quad S = 0, \] 
with $M_\gamma$ diagonally embedded, and classify the nilpotent orbits which give rise to hyperspherical varieties. 

This is a slightly stronger condition than the case we are considering in this paper; in the language of our proof, one would require that $H^\Delta\bs (L\times H)$ be a spherical $(L\times H)$-variety, and one could then obtain a classification along the same lines as in our proof. 

Their proof proceeds via a simple dimensionality criterion for hypersphericality, which, to the best of our knowledge, is not as readily applicable in our case, since the varieties in their case take the (relatively) simple form $G\times (f+\mf g^e)$. Furthermore, our proof (by the results of \cite{BZSV} used in Proposition \ref{prop:HypersphericalCriterion}) does not place any restriction on the choice of $S$ (which would also expand the number of possible nilpotent orbits).

Heuristically, since the coisotropic property corresponds to the multiplicity-one property for representations, one would expect that if one does not have multiplicity-one for a $(H,\rho)$-isotypic subspace of the generalised Whittaker representation $W_{\gamma,\psi}$ (which is the case we are considering in this paper), then one would not expect multiplicity-one for $W_{\gamma,\psi}$ as a $G\times H$-module. This explains the sense in which the condition in \cite{FU} is stronger. 

Notably, our classification includes the Shalika case as well as several larger `exceptional' cases, whereas the classification in \cite{FU} essentially gives the `hook-type' case and several smaller `exceptional' cases (for which the corresponding generalised Whittaker models have been studied in \cite{WZ}). See Remark \ref{rem:NoIsotypic2} for some remarks on the `hook-type' case in this situation.

\end{rem}

\section{Theta correspondence}\label{sec:ThetaCorrespondence}

 In this section, we review the theory and results surrounding the theta correspondence, in preparation for the next section, where the expected functorial lifting of representations will be facilitated by the theta lift. 
 
 Throughout this paper, we have fixed a non-trivial unitary character $\psi : F\rightarrow \C^\times$.

\subsection{Howe duality}

Suppose $(G_1,G_2)$ is a type I reductive dual pair; for our purposes we will take $G_1,G_2$ the isometry groups of a split orthogonal vector space $(V_1,B_1)$ and a symplectic vector space $(V_2,B_2)$ (or vice versa). If $\dim V_1$ is odd then in what follows we understand that we will have to work with representations of ${\rm Mp}(V_2)$ instead of $\Sp(V_2)$. We suppose also that $G_1$ is the smaller group of the two. 

One may restrict the Weil representation $\omega_\psi$ of ${\rm Mp}(V_1\otimes V_2)$ (corresponding to the character $\psi$) to $G_1\times G_2$. For each $\pi\in \Irr (G_1)$ define the \textit{big theta lift} $\Theta_\psi(\pi)$ of $\pi$ by \[\Theta_\psi(\pi):=(\omega_\psi \otimes \pi^\vee)_{G_1},\] the maximal $G_1$-invariant quotient of $\omega_\psi\otimes \pi^\vee$. Since $\psi$ is fixed, in what follows we shall sometimes drop the subscript $_\psi$ where there is no danger of confusion.

It is known that $\Theta(\pi)$ is either zero or has finite length with a unique irreducible quotient, which we denote $\theta(\pi)$ and call the \textit{small theta lift} of $\pi$. In fact we summarise the key result as follows:

\begin{thm}\label{thm:HoweDuality} (Howe duality)
Let \[ C = \{ (\pi_1,\pi_2) \in \Irr(G_1)\times \Irr(G_2) \mid \pi_1\otimes\pi_2\text{ is a quotient of }\omega_\psi \}. \]
Then $C$ is the graph of a \textit{bijective} (partially-defined) function between $\Irr(G_1)$ and $\Irr(G_2)$. 

Furthermore, we have \[ \dim \Hom(\omega_\psi, \pi_1\otimes\pi_2) \le 1\] for all $\pi_1\in \Irr(G_1), \pi_2\in \Irr(G_2)$. 
\end{thm}

We note also that by results of \cite{Sa2}, Howe duality may also be formulated in the $L^2$-setting. He showed in particular that the spectral support of $\hat{\omega_\psi}$ (unitary completion of $\omega_\psi$) is contained in the tempered dual of $G_1$.

\subsection{Functoriality} As mentioned, the theta lift is, for us, a means to realise the expected functorial lifting of representations from $G_1$ to $G_2$. We now make more precise what this means. Recall that we are assuming that $G_1$ is the smaller group of the two, and that we are working with split orthogonal vector spaces throughout. 

The central result in this regard is Adams' conjecture; by the recent results of \cite{BH}, we have

\begin{thm}\label{thm:Adam} (Adams' conjecture)
Suppose $\dim V_1$ and $\dim V_2$ are even. Suppose $\pi\in \Irr(G_1)$ has corresponding A-parameter $\phi$. Then for all \textit{sufficiently large} $\dim V_2$, the theta lift $\theta_\psi(\pi)\in \Irr(G_2)$ (is non-zero and) has corresponding A-parameter \[ \phi \oplus W_{\dim V_2 - \dim V_1 - 1}\] where $W_{\dim V_2 - \dim V_1 - 1}$ is the irreducible $\SL_2$-representation of dimension $(\dim V_2 - \dim V_1 - 1)$.
\end{thm}

Here $\dim V_2$ is \textit{sufficiently large} in the sense that it is at least the larger of two ``first occurrence indices" attached to the dual pair and $\pi$ (and $\psi$); to avoid excessive technicality, we do not define this precisely here, but suffice it to note that this will not pose any problems in the cases we will consider in Section \ref{sec:HookType}, in particular for generic representations $\pi$. We refer the reader to \cite{BH} (or other treatments of the standard theory of theta correspondence, such as \cite{Ga}), for details on the first occurrence indices. 

Finally, we note here that although Adams' conjecture is proven in \cite{BH} in the case where $\dim V_1,\dim V_2$ are even, it is expected in \cite{BH} that the same results will hold in all cases. 

\subsection{Transfer of nilpotent orbits}\label{sec:MomentMap} One would like to use the theta correspondence to relate two generalised Whittaker models on each member of a dual pair. To do so, one needs a correspondence of nilpotent orbits, and this is facilitated by a double fibration via moment maps. We refer to \cite{GZ}, \cite{Zh} for the details (cf. also \cite{DKP}) and state only the result here.

For the sake of simplifying notation, let us henceforth replace $G_1$ and $G_2$ with $G$ and $G'$ respectively, and similar for all other notation. 

\begin{prop}\label{prop:TransferNilpOrbit} (Transfer of nilpotent orbits via moment map)

One has moment maps \[ \g \xleftarrow[]{\phi} \Hom(V',V) \xrightarrow[]{\phi'} \g' \] 
defined by \[\phi(f) = f f^\ast,\] \[\phi'(f) = f^\ast f,\] where $f^\ast$ denotes the adjoint of the linear map $f$. 

Given a nilpotent element $e$ in the image of $\phi$ corresponding to a $\mfsl_2$-triple $\gamma$, one may uniquely define a nilpotent orbit/conjugacy class of $\mfsl_2$-triple $\gamma'$ of $\g'$ (with corresponding nilpotent element $e'\in\g'$) such that:
\begin{itemize}
    \item $e,e'$ are the images of some common element $f\in\Hom(V',V)$;
    \item the form on $V'$ restricts to a nondegenerate form on $\ker f$ (including if $\ker f = 0$),
    \item and $f$ sends the $k$-weight space of $V'$ to the $(k+1)$-weight space of $V$ for all $k\in \Z$ (here the weight spaces are under the $\mfsl_2$ action coming from $\gamma,\gamma'$). 
\end{itemize}

The partitions corresponding to $\gamma,\gamma'$ are related in the following way: suppose their corresponding Young tableaux are $d,d'$ respectively. Then one removes the first column of $d$ and adds suitably many rows of length 1, to obtain $d'$. 

Furthermore, recall from Section \ref{NilpPartition} that the nilpotent orbits of $\g,\g'$ are parameterised also by (symplectic or orthogonal) forms $B_j, B'_j$ on the multiplicity spaces $V_j, V'_j$. To obtain the corresponding forms $B'_j$ for $\gamma'$, the forms $B_j$ from $\gamma$ are left unchanged, (and the form $B'_1$ corresponding to the rows of length 1 in $d'$ is determined by the compatibility condition of Section \ref{NilpPartition}). 

In other words, one has \[ (V'_j,B'_j)  = (V_{j+1},B_{j+1}) \quad \text{for $j \geq 2$} \]
and \[ V'_1 =  V_2 \oplus V_{new} \] for a subspace $V_{new}$ corresponding to the \textit{newly added} rows of length 1 in $d'$. In fact, $V_{new}=\ker f$.

See Figure \ref{fig:TransferNilpOrbits} for an illustration. 

\begin{figure}[!h]
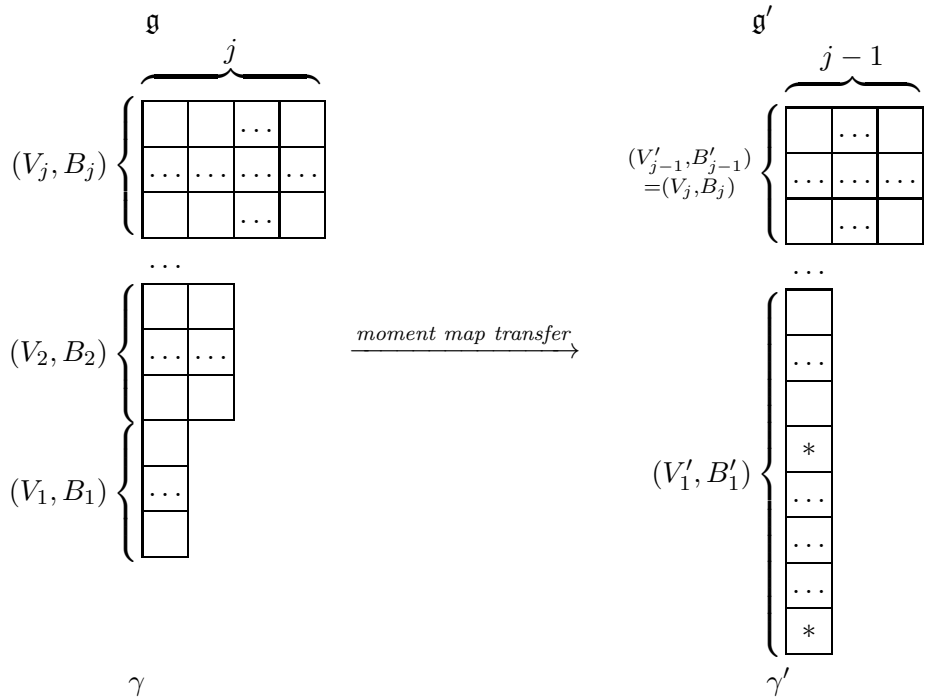

    \begin{center}
$\mf{g} \quad\quad\quad\quad\quad\quad\quad\quad\quad\quad\quad\quad\quad\quad\quad\quad\quad\quad\quad\quad \mf{g}'$
    
   \ytableausetup {mathmode,boxframe=normal,boxsize=1.5em,centertableaux} 
   \begin{tabular}{r@{}l}
   & $\overbrace{\hspace{6em}}^{\displaystyle j}$\\
   \begin{tabular}{r@{}l} $(V_j, B_j)\left\{\vphantom{\begin{ytableau}
    \quad \\ \quad  \\ \quad
 \end{ytableau}}\right.$ \\ \vspace{0.5em} \\
 $(V_2, B_2)\left\{\vphantom{\begin{ytableau}
    \quad \\ \quad  \\ \quad
 \end{ytableau}}\right.$ \\ $(V_1, B_1)\left\{\vphantom{\begin{ytableau}
    \quad \\ \quad  \\ \quad
 \end{ytableau}}\right.$ \end{tabular} & \begin{ytableau} \quad & \quad &\dots &\quad \\ \dots & \dots &\dots &\dots \\ \quad & \quad &\dots &\quad \\ \none[\dots] \\ \quad & \quad \\ \dots &\dots \\ \quad &\quad \\ \quad \\ \dots \\ \quad  \end{ytableau} \\ \\ \\ \\
   \end{tabular} $\xrightarrow{\text{moment map transfer}}$
\ytableausetup {mathmode,boxframe=normal,boxsize=1.5em,centertableaux} 
   \begin{tabular}{r@{}l} 
   & $\overbrace{\hspace{4.5em}}^{\displaystyle j-1}$\\
   \begin{tabular}{r@{}l} $\substack{(V'_{j-1}, B'_{j-1}) \\ = (V_j, B_j)}  \left\{\vphantom{\begin{ytableau}
    \quad \\ \quad  \\ \quad
 \end{ytableau}}\right.$ \\ \vspace{0.5em} \\
 $(V'_1, B'_1)\left\{\vphantom{\begin{ytableau}
    \quad \\ \quad  \\ \quad \\ \quad \\ \quad  \\ \quad \\ \quad \\ \quad
 \end{ytableau}}\right.$  \end{tabular} & \begin{ytableau} \quad &\dots &\quad \\  \dots &\dots &\dots \\  \quad &\dots &\quad \\ \none[\dots] \\   \quad \\ \dots \\ \quad \\ \ast \\ \dots \\ \dots \\ \dots \\ \ast  \end{ytableau}
   \end{tabular}

$\gamma \quad\quad\quad\quad\quad\quad\quad\quad\quad\quad\quad\quad\quad\quad\quad\quad\quad\quad\quad\quad\quad \gamma'$

\end{center}
\vspace{1em}
\caption{An illustration of the transfer of nilpotent orbits via the moment maps, in terms of the Young tableaux $d$ and $d'$. \\ $\ast$ indicates the newly added rows of length 1 in $d'$. }\label{fig:TransferNilpOrbits}
\end{figure}

\end{prop}
    
One verifies that the regular nilpotent orbit of $\g$ is in the image of the moment map $\phi$. 

\subsection{Results of Gomez-Zhu} We now come to the result of Gomez and Zhu \cite{GZ},\cite{Zh} which relates two generalised Whittaker models via the theta correspondence. Retain the notation of the previous sections. 

For simplicity also assume that the nilpotent orbit defined by $\gamma$ is in the image of the moment map $\phi$. 

Recall from Section \ref{sec:NilpOrbits} that \[ M_{\gamma}  \cong \prod_{k=1}^j G(V_k,B_k) \] and \[ M_{\gamma'}  \cong \prod_{k=1}^j G'(V'_k,B'_k). \]

In particular $M_{\gamma}$ and $M_{\gamma'}$ contain respectively factors 
\[G(V_1,B_1)\] and \[G'(V'_1,B'_1), \] corresponding respectively to the rows of length 1 in $d$ and $d'$. 

Furthermore $G'(V'_1,B'_1)$ contains a subgroup $G'(V_{new})$, which is an isometry group of the subspace $V_{new}\subseteq V'_1$ corresponding to the \textit{newly added} rows of length 1 in $d'$ (cf. Proposition \ref{prop:TransferNilpOrbit}).

In almost all of the cases we will work with, $d$ has no rows of length 2, so that in fact $G'(V_{new}) = G'(V'_1,B'_1)$. The only exception is when $d$ is simply the partition $[2]$, corresponding to the regular nilpotent orbit of $\mf{sp}_2$. 

We have that $G(V_1,B_1)$ and $G'(V_{new})$ forms a reductive dual pair (inside $\Sp(V_1\otimes V_{new})$, and of the same type as $G$ and $G'$ respectively). 

\begin{prop}\label{prop:MainGomezZhu} (\cite[Theorem 3.7]{Zh}, \cite[Theorem 6.2]{GZ})
For any $\pi\in \Irr(G')$, and for any genuine representation $\tau \in \Irr(\widetilde{G(V_1,B_1)})$, one has \[ W_{\gamma,\tau,\psi}(\Theta_\psi(\pi)) \cong W_{\gamma',\Theta(\tau)^\vee,\psi}(\pi^\vee).\]
Here:
\begin{itemize}
    \item $\Theta(\pi)$ is the big theta lift for the dual pair $(G,G')$;
    \item $\Theta(\tau)^\vee$ is the (dual of) the big theta lift for the dual pair $(G(V_1,B_1), G'(V_{new}))$. 
    \item Note that the isomorphism is also as $\widetilde{\big(G(V_2,B_2)\times\dots\times G(V_j,B_j)\big)}$-modules. 
    
\end{itemize}
\end{prop}

\begin{rem}
If the nilpotent orbit defined by $\gamma$ is \textit{not} in the image of the moment map $\phi$, then one has that \[ W_{\gamma,\tau,\psi}(\Theta_\psi(\pi)) = 0 \] for all $\pi\in \Irr(G')$. 

\end{rem}

\section{`Hook-type' partitions}\label{sec:HookType}

In this section, we determine the expected hyperspherical dual for the hook-type partitions $[n-a_1,1^{a_1}]$ of $\O_n$. As laid out in Section \ref{sec:DualityExpectations}, we show that the expectations at the (smooth) local level are satisfied. 

\subsection{Even orthogonal groups}\label{sec:EvenOrthogonal}

Suppose $n=2k$ is even. Then we must have $a_1=2a+1$ odd. 

\begin{thm2}\label{thm2:EvenOrthogonal}
The hyperspherical varieties $M_1,M_2$ defined respectively by 
\begin{itemize}
    \item the datum $\O_{2a+1}\times \SL_2\rightarrow \O_{2k}$, corresponding to the nilpotent orbit with partition $[2k-2a-1,1^{2a+1}]$, and trivial $S$;\\
    
    \item and the datum $\O_{2k-2a+1}\times \SL_2\rightarrow \O_{2k}$, corresponding to the nilpotent orbit with partition $[2a-1, 1^{2k-2a+1}]$, and trivial $S$.
\end{itemize}
are dual under the expected \cite{BZSV}-duality of hyperspherical varieties. 
    
\end{thm2}

We refer to this statement as ``Theorem$^\prime$", since there is as yet no formal definition of the statement ``$M_1$ and $M_2$ are hyperspherical dual". Rather, what one has is a list of expected properties, as in Expectation \ref{ex:relLanglandsQuantizationSmooth}, and the meaning of the assertion in Theorem$^\prime$ \ref{thm2:EvenOrthogonal} is given by the precise Theorem \ref{thm:EvenOrthogonalMain}. Similar remarks apply for the subsequent sections.  

Recalling the setup of Section \ref{sec:MomentMap}, let $\gamma_1$ be the nilpotent orbit of $\mf{so}_{2k}$ corresponding under the moment map (as in Proposition \ref{prop:TransferNilpOrbit}) to a regular nilpotent orbit $\gamma_{r,1}$ of $\mf{sp}_{2k-2a}$; we know that it corresponds to the partition $[2k-2a-1,1^{2a+1}]$.

\begin{figure}[H]
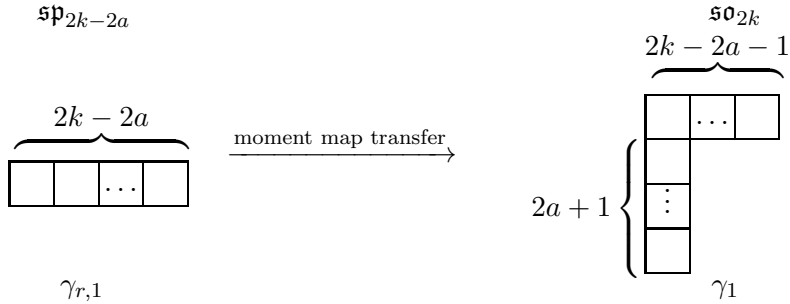

    \begin{center}
$\mf{sp}_{2k-2a} \quad\quad\quad\quad\quad\quad\quad\quad\quad\quad\quad\quad\quad\quad\quad\quad\quad\quad\quad\quad \mf{so}_{2k}$
    
   \ytableausetup {mathmode,boxframe=normal,boxsize=1.5em,centertableaux} 
   \begin{tabular}{r@{}l}
   $\overbrace{\hspace{6em}}^{\displaystyle 2k-2a}$\\
   \begin{ytableau} \quad & \quad &\dots &\quad  \end{ytableau}
   \end{tabular} $\quad\xrightarrow{\text{moment map transfer}}\quad$
\ytableausetup {mathmode,boxframe=normal,boxsize=1.5em} 
\begin{tabular}{r@{}l}
& $\overbrace{\hspace{4.5em}}^{\displaystyle 2k-2a-1}$\\
\begin{tabular}{r@{}l} \vspace{0.5em} \\
 $2a+1\left\{\vphantom{\begin{ytableau}
    \quad \\ \quad  \\ \quad
 \end{ytableau}}\right.$ \end{tabular} & \begin{ytableau} \quad&\dots &\quad \\ \quad \\ \vdots\\ \quad  \end{ytableau} 
\end{tabular}

$\gamma_{r,1} \quad\quad\quad\quad\quad\quad\quad\quad\quad\quad\quad\quad\quad\quad\quad\quad\quad\quad\quad\quad\quad \gamma_1$

\end{center}
\vspace{1em}
\caption{The Young tableaux corresponding to $\gamma_{r,1}$ and $\gamma_1$ respectively.  }
\end{figure}

Similarly, let $\gamma_2$ be the nilpotent orbit of $\mf{so}_{2k}$ corresponding under the moment map to a regular nilpotent orbit $\gamma_{r,2}$ of $\mf{sp}_{2a}$; we know that it corresponds to the partition $[2a-1, 1^{2k-2a+1}]$.

\begin{figure}[H]
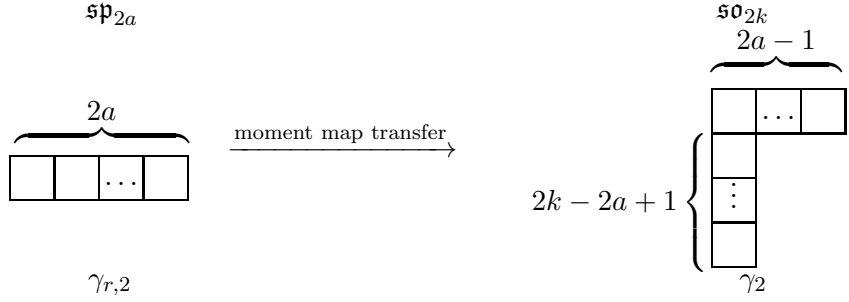

    \begin{center}
$\mf{sp}_{2a} \quad\quad\quad\quad\quad\quad\quad\quad\quad\quad\quad\quad\quad\quad\quad\quad\quad\quad\quad\quad \mf{so}_{2k}$
    
   \ytableausetup {mathmode,boxframe=normal,boxsize=1.5em,centertableaux} 
   \begin{tabular}{r@{}l}
   $\overbrace{\hspace{6em}}^{\displaystyle 2a}$\\
   \begin{ytableau} \quad & \quad &\dots &\quad  \end{ytableau}
   \end{tabular} $\quad\xrightarrow{\text{moment map transfer}}\quad$
\ytableausetup {mathmode,boxframe=normal,boxsize=1.5em} 
\begin{tabular}{r@{}l}
& $\overbrace{\hspace{4.5em}}^{\displaystyle 2a-1}$\\
\begin{tabular}{r@{}l} \vspace{0.5em} \\
 $2k-2a+1\left\{\vphantom{\begin{ytableau}
    \quad \\ \quad  \\ \quad
 \end{ytableau}}\right.$ \end{tabular} & \begin{ytableau} \quad&\dots &\quad \\ \quad \\ \vdots\\ \quad  \end{ytableau} 
\end{tabular}

$\gamma_{r,2} \quad\quad\quad\quad\quad\quad\quad\quad\quad\quad\quad\quad\quad\quad\quad\quad\quad\quad\quad\quad\quad \gamma_2$
\end{center}
\vspace{1em}
\caption{The Young tableaux corresponding to $\gamma_{r,2}$ and $\gamma_2$ respectively.  }\label{fig:EvenOrthogonalM2Transfer}
\end{figure}

Furthermore, as we have seen in Section \ref{sec:RelativeLanglandsDuality}, $M_1$ and $M_2$ have quantizations which are respectively the generalised Whittaker representations $W_{\gamma_1, \triv_1,\psi}$ and $W_{\gamma_2,\triv_2,\psi}$. Here $\triv_1,\triv_2$ are the trivial representations for the subgroups $H_1=\O_{2a+1}$ and $H_2=\O_{2k-2a+1}$ respectively. 

\begin{thm}\label{thm:EvenOrthogonalMain}
We have:
\begin{itemize}
\item If $\pi$ is an irreducible representation of $\O_{2k}$ which occurs as a quotient of $W_{\gamma_1,\triv_1,\psi}$, then $\pi=\theta_\psi(\sigma)$ for $\sigma$ an irreducible representation of $\Sp_{2k-2a}$.

Conversely, if $\sigma$ is an irreducible $\psi$-generic representation of $\Sp_{2k-2a}$, then $\pi:=\theta_\psi(\sigma)$ is an irreducible representation of $\O_{2k}$ which occurs as a quotient of $W_{\gamma_1,\triv_1,\psi}$. \\

\item If $\pi$ is an irreducible representation of $\O_{2k}$ which occurs as a quotient of $W_{\gamma_2,\triv_2,\psi}$, then $\pi=\theta_\psi(\sigma)$ for $\sigma$ an irreducible representation of $\Sp_{2a}$.

Conversely, if $\sigma$ is an irreducible $\psi$-generic representation of $\Sp_{2a}$, then  $\pi:=\theta_\psi(\sigma)$ is an irreducible representation of $\O_{2k}$ which occurs as a quotient of $W_{\gamma_2,\triv_2,\psi}$.

\end{itemize}
\end{thm}

\begin{proof}
We show the first statement; the second (the `dual statement') is similar. 

The main result of \cite{Zh} (cf. also \cite{GZ}), as stated in Proposition \ref{prop:MainGomezZhu}, implies that \[ W_{\gamma_{r,1},\triv,\psi}(\Theta_\psi(\pi)) \cong W_{\gamma_1,\triv_1,\psi}(\pi^\vee) \] for all $\pi\in \Irr(\O_{2k})$. (Here $\triv$ is the trivial representation for the trivial subgroup of $M_{\gamma_{r,1}}$. Henceforth we shall drop the $_{\triv}$ subscripts for ease of reading.)

On one hand, if $\pi$ occurs as a quotient of $W_{\gamma_{1}}$, then $W_{\gamma_1}(\pi^\vee) \ne 0$, hence $W_{\gamma_{r,1}}(\Theta(\pi))\ne 0$; in particular $\Theta(\pi)\ne 0$, hence $\theta(\pi)\ne 0$. In view of Theorem \ref{thm:HoweDuality}, $\pi$ is the (small) theta lift of an irreducible representation of $\Sp_{2k-2a}$. 

Note that if $\Theta(\pi)$ is itself already irreducible (i.e. equal to $\theta(\pi)$), which one expects to be true most of the time, then we have shown that $\pi$ is the (small) theta lift of an irreducible \textit{$\psi$-generic} representation of $\Sp_{2k-2a}$. 

On the other hand, let $\sigma$ be an irreducible $\psi$-generic (tempered) representation of $\Sp_{2k-2a}$, so $W_{\gamma_{r,1}}(\sigma)\ne 0$. As in the remarks after Theorem 4.1 of \cite{Ba}, since ($1\le a\le k-1$ and) $\sigma$ is generic, one has $\theta(\sigma)\ne 0$. We want to show that $W_{\gamma_1}(\theta(\sigma))\ne 0$. Suppose otherwise that $W_{\gamma_1}(\theta(\sigma))= 0$, then $W_{\gamma_{r,1}}(\Theta(\theta(\sigma)))= 0$; but this means that $\sigma$, as a quotient of $\Theta(\theta(\sigma))$ (cf. Theorem \ref{thm:HoweDuality}), is not generic, a contradiction. 
\end{proof}

Note that, by the (enhanced) Shahidi's conjecture \cite{HLL} for symplectic groups, any irreducible generic tempered representation lives in exactly one A-packet, and an A-packet is tempered if and only if it has a generic member. 

Since the theta-lift realises the respective functorial lifting via the maps $\O_{2k-2a+1}\times \SL_2\rightarrow \O_{2k}$ and $\O_{2a+1}\times \SL_2\rightarrow \O_{2k}$ (cf. Theorem \ref{thm:Adam}), we see that Theorem \ref{thm:EvenOrthogonalMain} realises the expected functorial lifting of Expectation \ref{ex:relLanglandsQuantizationSmooth}.

In other words, Theorem \ref{thm:EvenOrthogonalMain} shows Theorem$^\prime$ \ref{thm2:EvenOrthogonal} at the smooth local level.

\begin{rem}
From the proof, we see also that the multiplicity-one property should hold for $W_{\gamma_1}$ (and $W_{\gamma_2}$), that is, 
\[ \dim \Hom (W_{\gamma_1}, \pi) \le 1\]
for $\pi\in \Irr(\O_{2k})$, as long as $\Theta(\pi)$ is irreducible (i.e. equal to $\theta(\pi)$). 
\end{rem}

\begin{rem}
One has $1\le a\le k-1$. When $a=k-1$, the corresponding nilpotent orbit is trivial, and we recover (in principle) the case of the spherical variety $\O_{2k-1}\bs \O_{2k}$, which was studied in \cite{GaWa}. 
\end{rem}

\subsection{Odd orthogonal groups}

Suppose $n=2k+1$ is odd. Then we must have $a_1=2a$ even. In this case, the dual group $G^\vee$ is a symplectic group, so there are slight differences from the previous section. 

\begin{thm2}\label{thm2:OddOrthogonal}
The hyperspherical varieties $M_1,M_2$ defined respectively by 
\begin{itemize}
    \item the datum $\O_{2a}\times \SL_2\rightarrow \O_{2k+1}$, corresponding to the nilpotent orbit with partition $[n-2a,1^{2a}]$, and trivial $S$;\\
    
    \item and the datum $\Sp_{2k-2a+2}\times \SL_2\rightarrow \Sp_{2k}$, corresponding to the nilpotent orbit with partition $[2a-2, 1^{2k-2a+2}]$, and $S$ the standard symplectic representation of $\Sp_{2k-2a+2}$,
\end{itemize}
are dual under the expected \cite{BZSV}-duality of hyperspherical varieties.

\end{thm2}

Recalling the setup of Section \ref{sec:MomentMap}, let $\gamma_1$ be the nilpotent orbit of $\mf{so}_{2k+1}$ corresponding under the moment map to a regular nilpotent orbit $\gamma_{r,1}$ of $\mf{sp}_{2k-2a+2}$; we know that it corresponds to the partition $[2k-2a+1,1^{2a}]$. 

\begin{figure}[H]
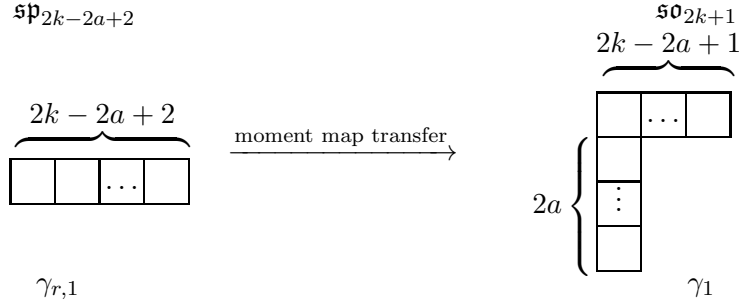

    \begin{center}
$\mf{sp}_{2k-2a+2} \quad\quad\quad\quad\quad\quad\quad\quad\quad\quad\quad\quad\quad\quad\quad\quad\quad\quad \mf{so}_{2k+1}$
    
   \ytableausetup {mathmode,boxframe=normal,boxsize=1.5em,centertableaux} 
   \begin{tabular}{r@{}l}
   $\overbrace{\hspace{6em}}^{\displaystyle 2k-2a+2}$\\
   \begin{ytableau} \quad & \quad &\dots &\quad  \end{ytableau}
   \end{tabular} $\quad\xrightarrow{\text{moment map transfer}}\quad$
\ytableausetup {mathmode,boxframe=normal,boxsize=1.5em} 
\begin{tabular}{r@{}l}
& $\overbrace{\hspace{4.5em}}^{\displaystyle 2k-2a+1}$\\
\begin{tabular}{r@{}l} \vspace{0.5em} \\
 $2a\left\{\vphantom{\begin{ytableau}
    \quad \\ \quad  \\ \quad
 \end{ytableau}}\right.$ \end{tabular} & \begin{ytableau} \quad&\dots &\quad \\ \quad \\ \vdots\\ \quad  \end{ytableau} 
\end{tabular}

$\gamma_{r,1} \quad\quad\quad\quad\quad\quad\quad\quad\quad\quad\quad\quad\quad\quad\quad\quad\quad\quad\quad\quad\quad \gamma_1$

\end{center}
\vspace{1em}
\caption{The Young tableaux corresponding to $\gamma_{r,1}$ and $\gamma_1$ respectively.  }
\end{figure}

Similarly, let $\gamma_2$ be the nilpotent orbit of $\mf{sp}_{2k}$ corresponding under the moment map to a regular nilpotent orbit $\gamma_{r,2}$ of $\mf{so}_{2a}$; we know that it corresponds to the partition $[2a-2, 1^{2k-2a+2}]$. 

Note that, in this case, there is some ambiguity in the choice of regular nilpotent orbit $\gamma_{r,2}$ of $\mf{so}_{2a}$ (since the corresponding partition is $[2a-1,1]$ and not $[2a]$). We fix a choice with the discriminant of the (orthogonal) form $B_1$ on the multiplicity space $V_1$ being trivial in $F^\times / F^{\times 2}$. 

\begin{figure}[H]
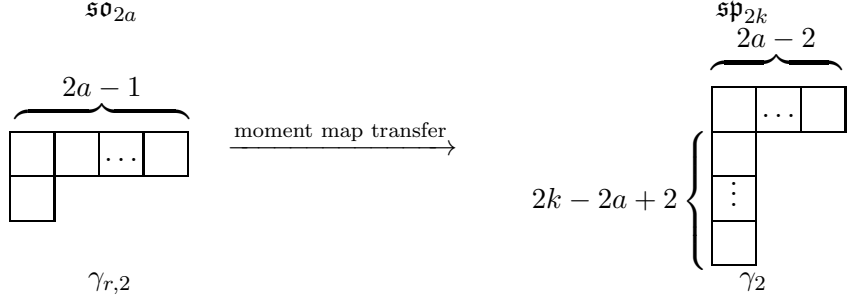

    \begin{center}
$\mf{so}_{2a} \quad\quad\quad\quad\quad\quad\quad\quad\quad\quad\quad\quad\quad\quad\quad\quad\quad\quad\quad\quad \mf{sp}_{2k}$
    
   \ytableausetup {mathmode,boxframe=normal,boxsize=1.5em,centertableaux} 
   \begin{tabular}{r@{}l}
   $\overbrace{\hspace{6em}}^{\displaystyle 2a-1}$\\
   \begin{ytableau} \quad & \quad &\dots &\quad \\ \quad  \end{ytableau}
   \end{tabular} $\quad\xrightarrow{\text{moment map transfer}}\quad$
\ytableausetup {mathmode,boxframe=normal,boxsize=1.5em} 
\begin{tabular}{r@{}l}
& $\overbrace{\hspace{4.5em}}^{\displaystyle 2a-2}$\\
\begin{tabular}{r@{}l} \vspace{0.5em} \\
 $2k-2a+2\left\{\vphantom{\begin{ytableau}
    \quad \\ \quad  \\ \quad
 \end{ytableau}}\right.$ \end{tabular} & \begin{ytableau} \quad&\dots &\quad \\ \quad \\ \vdots\\ \quad  \end{ytableau} 
\end{tabular}

$\gamma_{r,2} \quad\quad\quad\quad\quad\quad\quad\quad\quad\quad\quad\quad\quad\quad\quad\quad\quad\quad\quad\quad\quad \gamma_2$

\end{center}
\vspace{1em}
\caption{The Young tableaux corresponding to $\gamma_{r,2}$ and $\gamma_2$ respectively.  }
\end{figure}

The centraliser of $\gamma_2$, $M_{\gamma_2}$, is isomorphic to the direct product of $\mu_2$ and $H$, the isometry group of a $(2k-2a+2)$-dimensional symplectic vector space over $F$. As in subsection \ref{sec:FourierJacobiInduction}, let $\rho$ be the dual of the Weil representation (associated to $\psi$) of the metaplectic cover of $H$, which is also isomorphic to the Weil representation associated to $\bar{\psi}(x):=\psi(-x)$.  

Furthermore, as we have seen in Section \ref{sec:RelativeLanglandsDuality}, $M_1$ and $M_2$ have quantizations which are respectively the generalised Whittaker representations $W_{\gamma_1,\triv_1,\psi}$ and $W_{\gamma_2,\rho,\psi}$. If working with hyperspherical datum over $F$ (cf. Remark \ref{rem:Rationality}), we negate the symplectic form on $S$, so that its quantization is the Weil representation associated to $\bar{\psi}$, which is $\rho$.

\begin{thm}\label{thm:OddOrthogonalMain}
We have:
\begin{itemize}
\item If $\pi$ is an irreducible representation of $\O_{2k+1}$ which occurs as a quotient of $W_{\gamma_1,\triv_1,\psi}$, then $\pi=\theta_\psi(\sigma)$ for $\sigma$ an irreducible representation of ${\rm Mp}_{2k-2a+2}$.

Conversely, if $\sigma$ is an irreducible $\psi$-generic representation of ${\rm Mp}_{2k-2a+2}$, then if $\pi:=\theta_\psi(\sigma) \ne 0$, it is an irreducible representation of $\O_{2k+1}$ which occurs as a quotient of $W_{\gamma_1,\triv_1,\psi}$. \\

\item If $\pi$ is an irreducible representation of $\Sp_{2k}$ which occurs as a quotient of $W_{\gamma_2,\rho,\psi}$, then $\pi=\theta_\psi(\sigma)$ for $\sigma$ an irreducible representation of $\O_{2a}$.

Conversely, if $\sigma$ is an irreducible $\psi$-generic\footnote{one should think of the dependence not as being on $\psi$, but rather on the choice of regular nilpotent orbit $\gamma_{r,2}$ of $\mf{so}_{2a}$ (as above), as is explained in \cite[Section 12]{GGP}} representation of $\O_{2a}$, then $\pi:=\theta_\psi(\sigma)$ is an irreducible representation of $\Sp_{2k}$ which occurs as a quotient of $W_{\gamma_2,\rho,\psi}$. 

\end{itemize}
\end{thm}

\begin{proof}
Since the proof will be largely the same as that of Theorem \ref{thm:EvenOrthogonalMain}, we do not reproduce it here. 
\end{proof}

Note here that, by the local Shimura correspondence \cite{GaSa}, irreducible representations of the metaplectic group are parameterised by L-parameters into the symplectic group. 

Now since the theta-lift is expected to realise the respective functorial lifting via the maps $\Sp_{2k-2a+2}\times \SL_2\rightarrow \Sp_{2k}$ and $\O_{2a}\times \SL_2\rightarrow \O_{2k+1}$ (cf. Theorem \ref{thm:Adam}), we see that Theorem \ref{thm:OddOrthogonalMain} should realise the expected functorial lifting of Expectation \ref{ex:relLanglandsQuantizationSmooth}. 

In other words, Theorem \ref{thm:OddOrthogonalMain} shows Theorem$^\prime$ \ref{thm2:OddOrthogonal} at the smooth local level. 

\begin{rem}
From the proof, we see also that the multiplicity-one property should hold for $W_{\gamma_1}$ (and $W_{\gamma_2,\rho}$), that is, 
\[ \dim \Hom (W_{\gamma_1}, \pi) \le 1\]
for $\pi\in \Irr(\O_{2k+1})$, as long as $\Theta(\pi)$ is irreducible (i.e. equal to $\theta(\pi)$). 
\end{rem}

\begin{rem}
One has $2\le a\le k$. When $a=k$, the corresponding nilpotent orbit is trivial, and we recover (in principle) the case of the spherical variety $\O_{2k}\bs \O_{2k+1}$, which was studied in \cite{GaWa}. 
\end{rem}

\subsection{Closing remarks} We close this section with several remarks. 

\begin{rem}\label{rem:SymplecticHookType}
It is of note that we do also exhaust all the possible hook-type partitions for the group $\Sp_{2n}$, modulo the occurrence of the Weil representation, or the non-trivial $S$. 

When $S$ is trivial, this corresponds to trivial $\rho$, so that (cf. the discussion of Section \ref{notevengamma}) we should be working with representations of the metaplectic cover ${\rm Mp}_{2n}$ in its quantization. In the language of \cite{BZSV}, the corresponding $M$ has anomaly, cf. Remark \ref{rem:Anomaly}, and is currently excluded from the expectations of duality. 

The choice of $S$, or of $\rho$ as above, corresponds to the canonical choice explained in subsection \ref{sec:FourierJacobiInduction}. In fact $S$ is the dual of the symplectic space $(\mf{u}/\mf{u}^+)$, and the resulting Hamiltonian variety $M_2$ is a (twisted) cotangent bundle over $HU^+\bs G$, whose quantization is an induction from a character, as in (\ref{defFJW}). 

So we obtain a fairly complete picture of the duality in the case of hook-type partitions for the orthogonal and symplectic groups. 
\end{rem}

\begin{rem}
The main result of Gomez and Zhu \cite{Zh}, \cite{GZ} used in the proof has natural global analogues, special cases of which have been studied in for example \cite{GRS}. It is possible that this will play a role in the study of the analogous global problem. 
\end{rem}

\begin{rem}
As alluded to in \S \ref{sec:IntroResult}(b), our results on the hook-type partitions in this section give rise to a correspondence \[ e \leftrightarrow e^\vee \] between nilpotent orbits of $G$ and $G^\vee$, which resembles a duality theory of nilpotent orbits. It is very natural to ask if this correspondence coincides with, for instance, the Barbasch-Vogan duality of nilpotent orbits \cite{BV}. 

Unfortunately, this does not appear to be the case. For instance, we have computed that the nilpotent orbit with partition $[3,1,1,1,1]$ in type $B_3$ corresponds under Barbasch-Vogan duality to the nilpotent orbit with partition $[4,2]$ in type $C_3$, which is not of hook-type. 

\end{rem}

\begin{rem}\label{rem:NoIsotypic2}
As mentioned in Remarks \ref{rem:NoIsotypic} and \ref{rem:NoIsotypic3}, one may also consider the generalised Whittaker models as representations of $G\times \tilde{M_\gamma}$, especially when dealing with the Bessel and Fourier-Jacobi models (i.e. hook-type partitions). 

However, the situation in this case (where we are dealing with the Bessel and Fourier-Jacobi models) as regards to hyperspherical duality is in fact somewhat `simpler' than the situation we deal with, as the spherical varieties thus obtained are \textit{strongly tempered}, in the sense that its dual datum satisfies $X^\vee = G^\vee\times \tilde{M_\gamma}^\vee$ \cite{WZ}. 

The dual in this case is essentially known \cite{FU} to be the standard symplectic vector space (acted upon by the corresponding dual groups $G^\vee$ and $\tilde{M_\gamma}^\vee$ which form a reductive dual pair), quantizing to the Weil representation (or theta correspondence / Howe duality, as studied in \cite{Sa2}). The dual problem concerning the spectral decomposition of the Weil representation is essentially precisely Adams' conjecture (Theorem \ref{thm:Adam})!

We note that, for instance, corresponding problems related to the Bessel models have been studied by Liu \cite{Liu}, and for the Fourier-Jacobi models by Xue \cite{Xu}. 
\end{rem}

\section{Duality under symplectic reduction}\label{sec:DualitySymplecticReduction}

Following a suggestion of Venkatesh, let us now examine how duality in the $G\times M_\gamma$ case (Remark \ref{rem:NoIsotypic2}, \cite{FU}) may be seen to be related to duality in the $G$ case. In doing so, we will also see how hyperspherical duality is expected to relate to the operation of symplectic reduction, via the following guiding principle: \[ \textit{Hyperspherical duality `commutes' with symplectic reduction.} \]

We will now illustrate this using the case where $G$ is the even orthogonal group, that is, the case in Section \ref{sec:EvenOrthogonal}. Retain all notation of Section \ref{sec:EvenOrthogonal}. 

From (\ref{defWhittakerInduction2}), one easily sees that the hyperspherical variety $M_1$ is obtained by the symplectic reduction (Definition \ref{def:SymplecticReduction}), with respect to $H_1=\O_{2a+1}$, of the corresponding hyperspherical variety in the $G\times M_\gamma$ case, that is, the hyperspherical variety with datum \[  \O_{2a+1}^\Delta \times \SL_2 \rightarrow \O_{2k}\times \O_{2a+1} \quad \text{and} \quad S = 0. \] Denote this hyperspherical variety by $M_1'$. 

We may thus also view $M_1$ as the symplectic reduction of $M_1'$ with the trivial space $\{0\}$ for $\O_{2a+1}$. We have seen above that the dual of $M_1'$ is just the symplectic vector space $M_2' = \C^{2k}\otimes \C^{2a}$ given by the tensor product of the defining representations of the dual group $\O_{2k}\times \Sp_{2a}$. 

Now, since the dual of the trivial $\O_{2a+1}$-space $\{0\}$ is the Whittaker cotangent bundle $T^\ast (N,\psi \bs \Sp_{2a})$ for $N$ a maximal unipotent subgroup of $\Sp_{2a}$, one might expect, following a suggestion of Venkatesh, that the dual of $M_1$ is obtained by the symplectic reduction of $M_2'$ with the Whittaker cotangent bundle $T^\ast (N,\psi \bs \Sp_{2a})$ for $\Sp_{2a}$ (see Definition \ref{def:SymplecticReduction} and the remarks after it): \[ (M_2' \times_{\mf{sp}_{2a}^\ast} T^\ast (N,\psi \bs \Sp_{2a}))/\Sp_{2a}. \]

There are several ways to view such a reduction with the Whittaker cotangent bundle, which we will also call a `Whittaker reduction'. On one hand, writing $T^\ast (N,\psi \bs \Sp_{2a})$ in the usual form $(\lambda+\mf n^\perp)\times^N \Sp_{2a}$ (as in Example \ref{ex:TwistedCotangent}), one may easily view this reduction of $M_2'$ as a `reduction with respect to $(N,\psi)$'. (In fact, viewing the Whittaker cotangent bundle as a symplectic induction from $N$ to $\Sp_{2a}$, this is an instance of the symplectic analogue of Frobenius reciprocity, as in Remark \ref{rem:FrobeniusReciprocitySymplectic} and \cite[Theorem 3.4]{RZ}.) This form is most familiar to us from the representation-theoretic viewpoint, where it corresponds to the formation of $(N,\psi)$-coinvariants. 

On the other hand, one may again use (\ref{defWhittakerInduction2}) but for the Whittaker cotangent bundle, to write  \begin{equation}
\label{defWhittakerInduction3} 
T^\ast (N,\psi \bs \Sp_{2a}) \cong (f_{r,2} + \mf{sp}_{2a}^{e_{r,2}}) \times \Sp_{2a},
\end{equation} where $(f_{r,2} + \mf{sp}_{2a}^{e_{r,2}})$ is the principal Slodowy slice for the regular nilpotent orbit $\gamma_{r,2}$ of $\mf{sp}_{2a}$ (recall that we are continuing the notation of Section \ref{sec:EvenOrthogonal}). 

Now one sees readily that the desired reduction of $M_2'$ is then nothing but the pre-image, under the moment map $M_2'\rightarrow \mf{sp}_{2a}^\ast$, of the principal Slodowy slice $(f_{r,2} + \mf{sp}_{2a}^{e_{r,2}})$ of $\mf{sp}_{2a}$. See \cite[Proposition 3.9]{CR} for a detailed proof, where such a reduction is also called a \textit{Poisson slice} of $M_2'$ (with respect to $\Sp_{2a}$). 

Therefore one expects that the dual $M_2$, which is the cotangent bundle $T^\ast(H_2 U_2\bs \O_{2k})$ obtained by the Whittaker induction from the hook-type nilpotent orbit $\gamma_2$ of $\O_{2k}$, is precisely obtained by taking such a Poisson slice of the symplectic vector space $M_2'$.

We summarise the above discussion via the following diagram:

\[ \begin{tikzcd}[ampersand replacement=\&, row sep=large,column sep=huge, every label/.append
style={font=\normalsize}, outer sep=5pt]
M_1' \arrow{d}[swap]{\begin{array}{@{}c@{}}\text{symplectic reduction} \\ \text{with } \{0\} \end{array}}\arrow[r,leftrightarrow,"\text{duality}"] \& M_2' \arrow[d," \begin{array}{@{}c@{}}\text{Whittaker reduction,} \\ \text{or taking Poisson slice}\end{array}"] \\
M_1 \arrow[r,leftrightarrow,"\text{duality} "] \& M_2\\
\end{tikzcd} \]

\begin{prop}\label{prop:Slice}
Retain the notation of the preceding discussion. Assume that the Poisson slice, or Whittaker reduction, of $M_2'$ is hyperspherical. Then this Poisson slice is isomorphic to the dual $M_2$.
\end{prop}

\begin{proof}
First, one may use \cite[Corollary 3.4]{CR} to verify that the desired dimensions coincide: both have dimension $4ka-2a^2$. 

Recall that the nilpotent orbit $\gamma_2$ is precisely obtained by a transfer, via moment maps (Section \ref{sec:MomentMap}), from the regular nilpotent orbit $\gamma_{r,2}$ (cf. Figure \ref{fig:EvenOrthogonalM2Transfer}). Now, note that the moment maps in Section \ref{sec:MomentMap} are precisely the moment maps for the symplectic vector space $M_2'$ (after choosing some invariant identifications)! 

Therefore, retaining the notation of Section \ref{sec:MomentMap}, by definition of the transfer, one has an element $f\in M_2'$, which maps under the moment maps to the nilpotent elements $f_2$ of $\mf{so}_{2k}^\ast$ (corresponding to the nilpotent orbit $\gamma_2$), and $f_{r,2}$ of $\mf{sp}_{2a}^\ast$ (corresponding to $\gamma_{r,2}$) respectively. 

In particular, since its moment map image is $f_{r,2}$, $f$ lives in the Poisson slice of $M_2'$. Now since the moment map to $\mf{sp}_{2a}^\ast$ is a quotient map for $\O_{2k}$, it is also readily verified that $f$ is in the (unique) closed $\O_{2k}\times \mathbb{G}_{m}$-orbit of the Poisson slice, which is the pre-image of $f_{r,2}$ in $M_2'$. Here the $\mathbb{G}_{m}$-action is the grading (cf. Remark \ref{rem:Grading}) on the Poisson slice, which comes from the grading on $M_2'$ and the Whittaker cotangent bundle. Under the identification of (\ref{defWhittakerInduction3}), $\mathbb{G}_{m}$ acts on $\mf{sp}_{2a}^{e_{r,2}}$ so as to leave $0\in \mf{sp}_{2a}^{e_{r,2}}$ fixed. 

Furthermore, by the definition of $f$ as a linear map between vector spaces satisfying the conditions of Proposition \ref{prop:TransferNilpOrbit}, one has $\ker f = V_{new}$ is $(2k-2a+1)$-dimensional (and carries a non-degenerate orthogonal form). Now it is easy to verify that the stabiliser of $f$ is precisely \[ H_{2} = \O_{2k-2a+1}, \]  (cf. also \cite[Lemma 3.4]{Zh}). 

Now, since the moment map image of $f$ is $f_2 \in \mf{so}_{2k}^\ast$, one may employ the structure theorem of \cite{BZSV}, to deduce that the Poisson slice is indeed given by Whittaker induction along the datum \[ \O_{2k-2a+1}\times \SL_2 \rightarrow \O_{2k} \] corresponding to the nilpotent orbit $\gamma_2$, with the element $f$ as basepoint. The fact that the inducing symplectic vector space $S$ is trivial follows from the dimension comparison at the beginning.

\end{proof}

\begin{rem}
    This isomorphism is the geometric analog of the result of Gomez-Zhu (by the above interpretation of taking $(N,\psi)$-coinvariants). By the above proof, we also obtain a very conceptual geometric explanation of the involvement of the moment map transfer of nilpotent orbits, in the result of Gomez and Zhu. 
\end{rem}


Now, since $M_2$ is also a hook-type hyperspherical variety, one may repeat the above discussion with the roles of $M_1$ and $M_2$ reversed. In all, one has the following summarising diagram:

\[ \begin{tikzcd}[ampersand replacement=\&, row sep=huge,column sep=large, outer sep=5pt]
\O_{2k}\times \O_{2a+1} \circlearrowright M_1' \arrow{d}[swap]{\begin{array}{@{}c@{}}\text{symplectic reduction} \\ \text{with } \{0\} \end{array}}\arrow[r,leftrightarrow,"\text{duality}"] \&  M_2' = \C^{2k}\otimes \C^{2a} \circlearrowleft \O_{2k}\times \Sp_{2a}  \arrow[d," \begin{array}{@{}c@{}}\text{Whittaker reduction,} \\ \text{or taking Poisson slice}\end{array}"] \\
\O_{2k} \circlearrowright M_1 \arrow[r,leftrightarrow,"\text{duality} "] \& M_2 \circlearrowleft \O_{2k} \\ 
\O_{2k} \times \Sp_{2k-2a} \circlearrowright M_1'' = \C^{2k}\otimes \C^{2k-2a}  \arrow[u," \begin{array}{@{}c@{}}\text{Whittaker reduction,} \\ \text{or taking Poisson slice}\end{array}"]\arrow[r,leftrightarrow,"\text{duality} "] \& M_2'' \circlearrowleft \O_{2k} \times \O_{2k-2a+1} \arrow{u}[swap]{\begin{array}{@{}c@{}}\text{symplectic reduction} \\ \text{with } \{0\} \end{array}} \\
\end{tikzcd} \]

\section{Exceptional partitions}\label{sec:ExceptionalPartitions}

In this section, we examine  the `exceptional' partitions that feature in Theorems \ref{thm:OrthogonalClassification} and \ref{thm:SymplecticClassification}. In particular, we determine in many cases their expected hyperspherical duals and explain how the local Expectation \ref{ex:relLanglandsQuantizationSmooth}, as well as that for the `dual' problem (Remark \ref{rem:DualExpectation}), may be seen to be satisfied using (exceptional) theta correspondences. We also summarize some further expectations for hyperspherical duality at the end of the section; see  Expectation \ref{ex:Exceptional}.

In the following, we will take the liberty of working with different isogenous types of the groups involved, or with similitude groups. Note here one advantage of the proofs of Theorems \ref{thm:OrthogonalClassification} and \ref{thm:SymplecticClassification} (and the classification of spherical varieties in \cite{KVS}) is that they go through the commutator subalgebras of the Lie algebras (so that the proof carries over to isogenous or similitude groups without too much difficulty). For simplicity, we will thus not make much essential distinction between isogenous types or similitude groups. 
\vskip 5pt

\subsection{Orthogonal groups} Recall from Theorem \ref{thm:OrthogonalClassification} that the exceptional partitions which occur for the orthogonal group $\O_n$ are 
\[ [3,3], [4,4], [6,6].\]
We shall consider each of these in turn.
\vskip 5pt

\subsubsection{\bf $[3,3]$ partition}  Because $A_3 = D_3$,  we are essentially considering the general linear group $\GL_4$ in this case. Working in the context of $\GL_4$,  we are looking at the nilpotent orbit with hook-type partition $[3,1]$. In fact we may replace $\GL_4$ with any $\GL_n$ ($n\ge 2$), and a nilpotent orbit with partition $[n-1,1]$.

The hyperspherical dual is then expected to be (the cotangent bundle of) the standard representation of $\GL_n$, by the standard Jacquet-Shalika theory of integral representations of the standard L-function (applied to $\GL_n\times \GL_1$, with the trivial character on $\GL_1$). 

More precisely, we have the following meta-theorem

\begin{thm2}\label{thm2:Orthogonal33}
Let 
\begin{itemize}
    \item $M_1$ be the hyperspherical variety associated with the datum $\GL_1\times \SL_2\rightarrow \GL_n$  (corresponding to the nilpotent orbit $\gamma$ of $\GL_n$ with partition $[n-1,1]$) and trivial $S$;\\
    
    \item $M_2$ be the hyperspherical variety associated with the datum $\GL_n\times \SL_2\rightarrow \GL_n$ (corresponding to the trivial nilpotent orbit) and $S=\mathrm{std}\oplus\mathrm{std}^\ast$, where $\mathrm{std}$ is the standard representation of $\GL_n$.
\end{itemize}
Then $M_1$ and $M_2$ are dual under the expected \cite{BZSV}-duality.
 \end{thm2}

Let us now explicate the precise meaning of the above meta-theorem.
As we have seen in Section \ref{sec:RelativeLanglandsDuality}, the quantization of $M_1$ is the generalised Whittaker representation $W_{\gamma,\psi}$. For $M_2$, its quantization is (cf. Example \ref{ex:QuantizationWeil}) the pullback (after fixing a splitting) of the Weil representation $\omega_\psi$ of $\Sp_{2n}$ to the Levi factor $\GL_n$ of its Siegel parabolic subgroup. Then the mathematical content of Theorem' \ref{thm2:Orthogonal33} is:
\vskip 5pt

\begin{thm}\label{thm:Orthogonal33}
\begin{itemize}
    \item The irreducible representations of $\GL_n$ which occurs as a quotient of $W_{\gamma,\psi}$, are precisely the generic representations of $\GL_n$.\\

\item The irreducible representations of $\GL_n$ (of Arthur type) which occurs as a quotient of $\omega_\psi$ have Arthur parameters factoring through the morphism $\GL_1\times \SL_2\rightarrow \GL_n$ defining $M_1$. 

\end{itemize}
 
\end{thm}

\begin{proof}
Let us in fact view $\GL_n(F)$ as the unitary group $\U(V)$ for a Hermitian space $V$ of dimension $n$ over $E=F\times F$. In the unitary group case, almost all the results of Section \ref{sec:ThetaCorrespondence} have corresponding analogues (though we did not state them for simplicity); for details, we point the reader to \cite{GZ} and \cite{Zh}. 

One may then, exactly as in Section \ref{sec:HookType}, show that the irreducible representations of $\GL_n$ which occur as quotients of the corresponding generalised Whittaker model, are precisely $\psi$-generic representations of $\GL_n$. Note that the theta-lift for the (unitary) dual pair $\GL_n\times \GL_n$ in this case is precisely the identity $\pi\mapsto \pi$, cf. the remarks in \cite[Section 3.1]{Ga2}.

For the second statement (the dual problem), one may view the situation as concerning the theta correspondence for the (unitary) dual pair $\U_1\times \U_n$ in $\Sp_{2n}$, since the theta correspondence in this case involves the restriction of the Weil representation $\omega_\psi$ of $\Sp_{2n}$ to $\U_1\times \U_n$, and $\U_n\cong \GL_n$ is the group we are interested in. (Recall that irreducible representations of $\U_1\cong \GL_1$ are 1-dimensional.) The dual problem is then essentially just Adams' conjecture (cf. Theorem \ref{thm:Adam}) for the dual pair $\U_1\times \U_n$!
\end{proof}

\subsubsection{\bf $[4,4]$ partition}\label{44partition} 
In this case, we shall take $G = \PGSO_8$ to be the adjoint group, so that $G^{\vee} = \Spin_8$ is simply connected.  The phenomenon  of triality has some interesting structural implications in this context, which we shall first explicate. 
\vskip 5pt
\begin{itemize}
\item[(a)]  First, there are 3 non-conjugate homomorphisms
\[  f_j:  \SO_8 \longrightarrow G= \PGSO_8 \quad \text{and likewise} \quad  p_j: G^{\vee}= \Spin_8 \longrightarrow \SO_8. \]
The $f_j$'s and $p_j$'s are permuted cyclically by a triality automorphism $\theta$ of $\PGSO_8$ and $\Spin_8$ respectively.  
\vskip 5pt

Now the description of nilpotent orbits in type $D_4$ by partitions of $8$ makes sense only if one is working with $\SO_8$ with its standard representation. In the present setting, this signifies that we have distinguished one of the 3 conjugacy classes of maps, say $f_1$ and $p_1$, as the standard one. Thus, $p_1: \Spin_8 \rightarrow \SO_8$ is considered the standard representation, whereas $p_2$ and $p_3$ are considered the half-spin representations of $\Spin_8$. 
 \vskip 5pt
 
\item[(b)]   If one denotes by  $\SO_7$ the stabilizer in $\SO_8$ of a unit vector in the standard representation, then set
 \[  \Spin_7^{[j]}  := p_j^{-1}(\SO_7) \subset \Spin_8. \]
This gives 3 distinct conjugacy classes of embeddings $\Spin_7 \rightarrow \Spin_8$ and 3 spherical varieties
\[  X_j := \Spin_7^{[j]} \bs \Spin_8. \]
Having fixed $p_1: \Spin_8 \rightarrow \SO_8$ as the standard representation of $\Spin_8$, 
the restriction of $p_2$ or $p_3$ to $\Spin_7^{[1]}$ is then the  irreducible spin representation of $\Spin_7$.  Moreover, one has
\[  \Spin_7^{[i]} \cap \Spin_7^{[j]} \cong G_2 \quad \text{if $i \ne j$}, \]
where $G_2$ is the exceptional group of rank $2$.
\vskip 5pt

Likewise, the restriction of $f_j$ to $\SO_7$ gives 3 conjugacy classes of embedding $\SO_7 \longrightarrow  \PGSO_8$ permuted by triality.
\vskip 5pt

\item[(c)]  If one starts with a  nilpotent orbit of type $[5,1,1,1]$ relative to $(f_1, p_1)$ and apply the triality automorphism to it, one obtains a nilpotent orbit of type $[4,4]$ relative to $(f_1, p_1)$ (cf. \cite[Example 5.3.7]{CM}). 
\end{itemize}

After the above preliminaries on triality, we can now formulate the following meta-theorem.

\begin{thm2}\label{thm2:Orthogonal44}
Let
\begin{itemize}
    \item $M_1$ be the hyperspherical variety associated to the datum corresponding to a nilpotent orbit  of $\PGSO_8$ associated to a partition $[4,4]$ (relative to $f_1$);\\
    
    \item $M_2$ be the the cotangent bundle of the spherical variety
\[  X_2 = \Spin^{[2]}_7 \bs \Spin_8.\]
\end{itemize}
Then $M_1$ and $M_2$ are dual under the expected \cite{BZSV}-duality. 
 
\end{thm2}

As before, the precise mathematical meaning of this meta-theorem involves resolving two  branching problems.
Indeed, we shall see that the desired  result can be deduced 
by an application of a triality automorphism $\theta$ to the results of Section \ref{sec:EvenOrthogonal} in the case $a=3$ and $k=4$,  i.e. for the duality between the nilpotent orbit of type $[5,1,1,1]$ and  $\SO_7\bs \SO_8 \cong \Spin^{[1]}_7\bs \Spin_8 = X_1$ (with the isomorphism induced by $p_1$).   

\vskip 5pt
More precisely, since a triality automorphism $\theta$ carries $\Spin_7^{[1]}$ to $\Spin_7^{[2]}$, it carries the irreducible quotients of 
$C^\infty_c(X_1)$ to that of $C^\infty_c(X_2)$, via:
\[  C^\infty_c(X_2) \cong C^\infty_c(X_1)^{\theta}. \]
 By Theorem \ref{thm:EvenOrthogonalMain}, the irreducible quotients of $C^\infty_c(X_1)$ are described in terms of the dual data 
\[ \begin{CD}
 \SO_3 \times \SL_2 @>>> \SO_8 @>f_1>> \PGSO_8  \end{CD} \]
whose restriction to $\SL_2$ corresponds to the partition $[5,1,1,1]$ (relative to $f_1$). In view of (c), the composition of this map  with $\theta$  gives a new morphism whose restriction to $\SL_2$ corresponds to the partition $[4,4]$ (relative to $f_1$).  This implies that the solution to the branching problem for the quantization of $M_2$ (i.e. $C^\infty_c(X_2)$) is given in terms of the initial data defining $M_1$. 

\vskip 5pt

For the dual problem,  the branching problem for the quantization of $M_1$ is the generalized Whittaker model attached to the partition $[4,4]$. By (c), this is carried by the triality automorphism $\theta$ to the branching problem for the generalized Whittaker model attached to $[5,1,1,1]$.  This latter branching problem has been addressed in Theorem \ref{thm:EvenOrthogonalMain} and its solution is given in terms of the dual data
\[   \Spin^{[1]}_7 \longrightarrow \Spin_8. \]
(More precisely, the irreducible quotients of the generalized Whittaker model attached to the partition $[5,1,1,1]$ are classical theta lifts of generic representations of ${\rm PGSp}_6$). 
The composition of the above dual data  with $\theta$ produces 
\[  \Spin_7^{[2]} \longrightarrow \Spin_8. \]
Thus the solution of the branching problem for the quantization of $M_1$  is given by the above map, which is the initial data for $M_2$. 
\vskip 5pt

\vskip 5pt

\subsubsection{\bf $[6,6]$ partition}\label{66partition} One expects the following, by the results of \cite{WZ}: 

\begin{thm2}\label{thm2:Orthogonal44}
Let
\begin{itemize}
    \item $M_1$ br the hyperspherical variety associated with the datum corresponding to a nilpotent orbit $\gamma$ of $\PGSO_{12}$ with partition $[6,6]$;\\
    
    \item $M_2$ be the (multiplicity-free \cite{Kn2}) half-spin representation $S$ of $\Spin_{12}$.
\end{itemize}
Then $M_1$ and $M_2$ are dual under the expected \cite{BZSV}-duality. 
 \end{thm2}

As before, the mathematical content  of this meta-theorem involves resolving the branching problems associated to quantization of $M_1$ and $M_2$, in terms of the data defining the other. 
Let us sketch how this can be done in this case.
\vskip 5pt

The problem of determining the irreducible quotients of the quantization of $M_1$, or of the generalised Whittaker representation $W_{\gamma,\psi}$ arising from the $[6,6]$ partition  is essentially resolved by the results in \cite[Section 9]{WZ}, in which the local multiplicity for this model is studied.   This is an instance of the ``strongly tempered" case, in conformity with the fact that the data defining $M_2$ is the pair
\[  ( {\rm id}: \Spin_{12} \longrightarrow \Spin_{12},  S). \]

\vskip 5pt

Therefore what remains is to investigate the dual problem, which concerns the quantization of $M_2$, or of the half-spin representation of $\Spin_{12}$. This is the pullback, via the half-spin representation, of the Weil representation $\omega_\psi$ obtained from ${\rm Mp}_{32}$. (Note that the metaplectic cover splits over $\Spin_{12}$, a manifestation of the fact that $M_2$ is anomaly-free (Remark \ref{rem:Anomaly}).) Let us continue to denote this representation of $\Spin_{12}$  by $\omega_\psi$. 
\vskip 5pt

To analyze the irreducible quotients of $\omega_{\psi}$ as a $\Spin_{12}$-module, we will  make use of the dual pair $(\SL_2, H) := (\SL_2, \Spin_{12})$ in the (simply-connected) exceptional group of type $E_7$, where $H$ is the derived subgroup of the Levi factor $L$ of a Heisenberg parabolic $P=LU$ of $E_7$. The unipotent radical $U$ is a Heisenberg group corresponding to a 32-dimensional symplectic space, on which $H$ acts via the half-spin representation.  From the description \cite[Proposition 43]{Ru} of the mixed model of the minimal representation $\Pi$ of $E_7$, in terms of the Weil representation $\omega_\psi$ obtained on $H$, one has: as a representation of $\Spin_{12}$,
\[ \omega_\psi\cong \Pi_{N,\psi},\] 
for $N$ a maximal unipotent subgroup of $\SL_2$. 

Hence we see that the irreducible quotients of $\omega_\psi\cong\Pi_{N,\psi}$ consists of lifts (via this exceptional theta correspondence) from $\psi$-generic representations of $\SL_2$, which exhibits the desired lifting of Expectation \ref{ex:relLanglandsQuantizationSmooth}. It remains to check the $\SL_2$-type of the lifted representations (they should correspond to the $[6,6]$ partition). That this is the case follows from \cite[Proof of Prop. 4.4, Pg.1239, line -4]{GaSa2}.
\vskip 5pt

\subsection{Symplectic groups} Recall from Theorem \ref{thm:SymplecticClassification} that the exceptional partitions which occur for the symplectic group $\Sp_{2n}$ are \[ [3,3], [5,5], [3,3,1^{2a}], [5,5,1^{2a}]. \]

\subsubsection{} The models related to the partitions \[ [3,3,1^{2a}], [5,5,1^{2a}] \] for $a>0$ are closely related to that for the partitions \[ [4,4], [6,6] \] in the orthogonal group case, by the theta correspondence and results of Gomez and Zhu \cite{GZ}, \cite{Zh} (as in Section \ref{sec:ThetaCorrespondence}). In particular, once one understands the models for $[4,4], [6,6]$, then one also understands the models for $[3,3,1^{2a}], [5,5,1^{2a}]$. This allows one to easily formulate conjectural expectations for their hyperspherical duals; some such examples may be found in Expectation \ref{ex:Exceptional} below. 

However, the corresponding dual problems seem to be much more difficult, and we are not yet aware of any way to resolve this case in complete generality. 

\subsubsection{\bf $[3,3]$ partition}  \label{SS:33}
Similar to the orthogonal group case above in subsection \ref{44partition}, one has:

\begin{thm2}\label{thm2:Symplectic33}
Let
\begin{itemize}
    \item $M_1$ be the hyperspherical variety for $\PGSp_6$ associated to the datum corresponding to a nilpotent orbit  $\gamma$ of $\PGSp_6$ with partition $[3,3]$;\\
    
    \item $M_2$ be the cotangent bundle of the  spherical variety 
    \[G_2 \bs \Spin_7.\] 
    
\end{itemize}
Then $M_1$ and $M_2$  are duals of each other  under the expected \cite{BZSV} duality.
 \end{thm2}

The hyperspherical variety $M_1$ has a quantization which is the generalised Whittaker representation $W_{\gamma,\psi}$. Working with the group $\PGSp_6$ instead, let us consider the dual pair $(G_2, \PGSp_6)$ in the (adjoint) exceptional group of type $E_7$. Let $\Pi$ be the minimal representation of $E_7$, and $(N,\psi)$ a Whittaker datum for the group of type $G_2$. By \cite[Proposition 11.5]{GaSa3}, we see that 
\[ W_{\gamma,\psi}\cong \Pi_{N,\psi}. \] 
Therefore, since the Howe duality (and functorial properties of the theta-lift) has been shown for this dual pair $(G_2, \PGSp_6)$ in \cite{GaSa3}, one sees that the irreducible quotients of the generalised Whittaker representation $W_{\gamma,\psi}\cong \Pi_{N,\psi}$ corresponding to the $[3,3]$ partition, consists of lifts (via this exceptional theta correspondence) of generic representations of $G_2$. In other words, the L-parameters of these irreducible quotients factor through the inclusion
\[  G_2 \longrightarrow \Spin_7 \]
which is the data defining $M_2$. 
\vskip 5pt

For the dual problem,  the quantization of $M_2=T^\ast(G_2\bs \Spin_7)$  is $C^\infty_c(G_2\bs \Spin_7)$. One needs to show that the irreducible quotients of this representation have A-parameters which factor through
\[  \SO_3 \times \SL_2 \longrightarrow \SO_3 \times \SO_3 \longrightarrow \PGSp_6. \]
This means that these irreducible representations of $\Spin_7$ should be  functorial lifts from $\SL_2$.  This expectation can in fact  be shown by the classical theta correspondence for $\SL_2 \times \SO_8$. 

More precisely, as we saw in subsection \ref{44partition}, the theta correspondence for $\SL_2 \times \SO_8$ shows that the irreducible quotients of $C^\infty_c(\SO_7 \backslash \SO_8)$ are given by theta lifts from generic representations of $\SL_2$. 
On the other hand, by item (b) in \S \ref{44partition},  we have  an embedding:
\[  p_2: \Spin_7 \hookrightarrow \SO_8   \]
inducing
\[  G_2 \bs \Spin_7 \cong  \SO_7 \bs \SO_8.    \]
Hence the irreducible constituents of 
\[  C^\infty_c(G_2 \bs \Spin_7) \cong  C^\infty_c(\SO_7 \bs \SO_8) \]
should be the irreducible representations of $\Spin_7$ obtained by restriction, via $p_2$, of the theta lifts of irreducible generic representations of $\SL_2$. 


\begin{rem}
This $[3,3]$ case is the only exceptional case for the orthogonal or symplectic groups (and one of two exceptional cases in general) which features in the classification of \cite{FU}. Similar to Section \ref{sec:DualitySymplecticReduction}, one has that the dual of their $\PGSp_6\times \PGL_2$-variety $M_1'$ is the 16-dimensional symplectic vector space $M_2'$ for $\SO_8\times \SL_2$, restricted to $\Spin_7\times \SL_2$. As in Proposition \ref{prop:Slice}, one has that the $\SL_2$-Whittaker reduction, or Poisson slice, of $M_2'$, is isomorphic to the cotangent bundle of $\SO_7\bs\SO_8$, which is in turn isomorphic to the cotangent bundle of $G_2\bs\Spin_7$, which is $M_2$. 

Finally, let us remark also in passing, that the final exceptional case of \cite{FU} is a duality between $M_1'$, the $G_2\times \SL_2$-variety corresponding to the 8-dimensional nilpotent orbit of $G_2$, and $M_2'$, the 14-dimensional symplectic vector space for $\SO_7\times \SL_2$, restricted to $G_2\times \SL_2$ (which is in fact anomalous). Again, as in Section \ref{sec:DualitySymplecticReduction} and the previous paragraph, one has an expected duality between 
\begin{itemize}
    \item $M_1$ the hyperspherical variety for $G_2$ associated to the datum corresponding to the 8-dimensional nilpotent orbit $\gamma$ of $G_2$; 
    \item $M_2$ the cotangent bundle of the  spherical variety 
    \[\SL_3 \bs G_2 (\cong \SO_6\bs \SO_7).\] 
\end{itemize}
The corresponding branching problems can be shown exactly as for the $[3,3]$ case in this subsection, using the exceptional theta correspondence for the dual pair $\PGL_3\times G_2$ (in $E_6$), and the classical theta correspondence for $\widetilde{\SL_2}\times \SO_7$ (restricted to $\widetilde{\SL_2}\times G_2$), respectively. The former has been completely understood by \cite{GaSa3}, and the latter by \cite{GaGu}. 
\end{rem}

\subsubsection{\bf $[5,5]$ partition} Similar to the orthogonal group case above in subsection \ref{66partition}, one has:

\begin{thm2}\label{thm2:Symplectic55}
Let
\begin{itemize}
    \item $M_1$ be the hyperspherical variety associated with a nilpotent orbit $\gamma$ of $\PGSp_{10}$ with partition $[5,5]$;\\
    
    \item  $M_2$ be the (multiplicity-free \cite{Kn2}) spin representation $S$ of $\Spin_{11}$.
\end{itemize}
Then $M_1$ and $M_2$  are duals of each other  under the expected \cite{BZSV}-duality.
 
\end{thm2}

Similar to subsection \ref{66partition}, the problem of determining the irreducible quotients of the generalised Whittaker representation $W_{\gamma,\psi}$ arising from the $[5,5]$ partition is essentially resolved by the results in \cite[Section 9]{WZ}, in which the local multiplicity for this model is studied. This is an instance of the ``strongly tempered" case, whose  solution is expressed in terms of the dual data
\[  ( {\rm id}: \Spin_{11} \rightarrow \Spin_{11}, S). \]

Therefore what remains is to study the dual problem, which concerns the quantization of the spin representation $S$ of $\Spin_{11}$. But since the spin representation of $\Spin_{11}$ is the half-spin representation for $\Spin_{12}$ restricted to $\Spin_{11}$, this quantization is just the Weil representation of ${\rm Mp}_{32}$ pulled back to $\Spin_{12}$ (which we studied above in \S  \ref{66partition}) and then further to $\Spin_{11}$.  

\subsection{\bf A non-even example}
Let us give another example of the \cite{BZSV}-duality which involves the case of a non-even nilpotent orbit.
\vskip 5pt

\begin{thm2}
Let

\begin{itemize}
\item $M_1$ be the hyperspherical variety associated with a  (non-even) nilpotent orbit $\gamma$ of $\PGSO_8$ with partition $[2,2,1,1,1,1]$;\\

\item $M_2$ be the cotangent bundle of the spherical variety
\[   G_2 \backslash \Spin_8. \]
\end{itemize}
Them $M_1$ and $M_2$ are dual under the expected \cite{BZSV}-duality.
\end{thm2}
\vskip 5pt

To justify this meta-theorem, we first consider the branching problem associated to the quantization of $M_1$, which is a generalized Whittaker representation $W_{\gamma, \triv, \psi}$. 
By the similitude theta correspondence for $\PGSO_8 \times \PGSp_6$, and applying  the result of Gomez-Zhu for this dual pair, one sees that the irreducible quotients of $W_{\gamma, \triv,\psi}$ are theta lifts of the irreducible quotients of the generalized Whittaker representation of $\PGSp_6$ associated with a nilpotent orbit with partition $[3,3]$. As we have alluded to in subsection \ref{SS:33}, the latter irreducible quotients are themselves theta lifts of generic representations of $G_2$, via the exceptional theta correspondence for $G_2 \times \PGSp_6$. This shows that the irreducible quotients of the quantization of $M_1$ are those whose A-parameters factor through the map
\[  G_2 \longrightarrow \Spin_8 \]
which is the datum defining $M_2$.
\vskip 5pt

Next we consider the branching problem arising from the quantization of $M_2$, so that we are interested in determining the irreducible quotients of  $C^{\infty}_c(G_2 \bs \Spin_8)$.
For this, we shall make use of the exceptional theta correspondence  for the dual pair 
\[  \SL_2^3/ \mu_2^{\Delta} \times \Spin_8   \]
in the adjoint group of type $E_7$.  As discussed in \cite{GaGo}, the irreducible quotients of $C^{\infty}_c(G_2 \bs \Spin_8)$ are theta lifts of generic representations of $\SL_2^3$. 
A study of this theta correspondence should show that the A-parameters of these theta lifts factor through the map 
\[  H \times \SL_2 = (\SL_2 \times_{\mu_2} \SL_2 \times_{\mu_2} \SL_2) \times \SL_2  \longrightarrow \PGSO_8. \]
This is the datum defining $M_1$. 
\vskip 5pt

\subsection{\bf Some speculations} 
We conclude this paper with some further speculations about the \cite{BZSV}-duality, arranging some examples in families.  

\begin{claim}\label{ex:Exceptional} 
    
One has the following tables of examples of hyperspherical dual pairs, ignoring (for clarity) all issues of isogeny and center: 

\begin{figure}[H]
    \begin{center}
\ytableausetup {mathmode,boxframe=normal,boxsize=1em,centertableaux} 
      \begin{tabular}{|c|c|c|} \hline
   $M$ & $\PGSp_6$, \begin{ytableau}  \quad &\quad &\quad \\  \quad &\quad &\quad  \end{ytableau} & $\PGSO_{8}$, \begin{ytableau}  \quad &\quad \\  \quad &\quad \\ \quad \\ \quad \\ \quad \\ \quad \end{ytableau}    \\ \hline
   $M^\vee$ & $T^\ast(G_2\bs \Spin_7)$ & $T^\ast(G_2\bs \Spin_8)$ \\ \hline
   \end{tabular}
      \end{center}
      \vspace{-1em}
\end{figure}

\begin{figure}[H]
    \begin{center} 
   \begin{tabular}{|c|c|c|c|} \hline
   $M$ & $\PGSO_8$, \begin{ytableau} \quad & \quad &\quad &\quad \\ \quad & \quad &\quad &\quad  \end{ytableau} & $\PGSp_{8}$, \begin{ytableau} \quad & \quad &\quad \\ \quad & \quad &\quad \\ \quad \\ \quad  \end{ytableau} &$\PGSO_{10}$, \begin{ytableau} \quad & \quad \\ \quad & \quad \\ \quad \\ \quad \\ \quad \\ \quad \\ \quad \\ \quad  \end{ytableau}   \\ \hline
   $M^\vee$ & $T^\ast(\Spin_7\bs \Spin_8)$ & $T^\ast(\Spin_7\bs \Spin_9)$ & $T^\ast(\Spin_7\bs \Spin_{10})$ \\ \hline
   \end{tabular}

         \end{center}
               \vspace{-1em}
\end{figure}

\begin{figure}[H]
    \begin{center}

      \begin{tabular}{|c|c|} \hline
   $M$ & $\PGSO_{10}$, \begin{ytableau}  \quad &\quad &\quad &\quad \\ \quad & \quad &\quad &\quad \\ \quad \\ \quad \end{ytableau}     \\ \hline
   $M^\vee$ & $\Spin_{10}$, \text{spin rep}  \\ \hline
   \end{tabular} 
      \end{center}
\end{figure}
 \vspace{-1em}
   \begin{figure}[H]
    \begin{center}
   \begin{tabular}{|c|c|} \hline
   $M$ & $\PGSp_{10}$, \begin{ytableau}  \quad & \quad &\quad &\quad &\quad \\ \quad &\quad & \quad &\quad &\quad \end{ytableau}     \\ \hline
   $M^\vee$ & $\Spin_{11}$, \text{spin rep}  \\ \hline
   \end{tabular} 
        \end{center}
\end{figure}
 \vspace{-1em}
   \begin{figure}[H]
    \begin{center} 
   \begin{tabular}{|c|c|} \hline
   $M$ & $\PGSO_{12}$, \begin{ytableau}  \quad &\quad & \quad &\quad &\quad &\quad \\ \quad &\quad &\quad & \quad &\quad &\quad \end{ytableau}     \\ \hline
   $M^\vee$ & $\Spin_{12}$, \text{half-spin rep}  \\ \hline
   \end{tabular}

   \end{center}
\end{figure}

(The occurrence of a Young tableaux indicates the Whittaker induction or generalised Whittaker model corresponding to the nilpotent orbit with that Young tableaux.)
\end{claim}

\begin{rem}
    Several remarks: \begin{itemize}
        \item Within each table, the models are related by the result of Gomez-Zhu (Proposition \ref{prop:MainGomezZhu}), moving from left to right. The choice of $S$ should be dictated according to this result. 
        \item Note that the example involving the (multiplicity-free) spin representation for $\SO_{10}$ comes from the work of Ginzburg \cite{Gi} studying the Spin L-function of ${\rm GSO}_{10}$. 
        \item Finally, the above tables account for essentially all of the `exceptional' spherical varieties (to do with low-rank orthogonal groups) which occur in the classification of spherical varieties \cite{KVS}. 

    \end{itemize}
\end{rem}

\subsubsection{A family of dual pairs from spin representations} Let us now examine how the examples in the last three tables of Expectation \ref{ex:Exceptional} can in fact be seen to fit into a larger family of examples.

For $0\le j\le 5$, let $M_j^\vee$ denote the restriction of the half-spin representation of $\Spin_{12}$ to its subgroup \[ G_j^\vee := \Spin_{12-j}\times \Spin_j. \] Note that $M_j^\vee$ is hence either the tensor product of spin reps (when $j$ is odd), or the (sum of) tensor products of half-spin reps (when $j$ is even). 

Now we have the following expectation, or meta-theorem:

\begin{thm2}
When $j$ is even, then $G_j = \PGSO_{12-j} \times \PGSO_j$. 

Let $\gamma$ be the nilpotent orbit of $\PGSO_{12-j}$ with partition $[6-j, 6-j, 1^j]$; then $M_\gamma = (\GSp_2 \times \GSO_j)^{\text{sim}}/\mathbb{G}_m$. Here $(\GSp_2 \times \GSO_j)^{\text{sim}}$ consists of the pairs of elements of $\GSp_2$ and $\GSO_j$ with the same similitude character. 

Then the hyperspherical dual $M_j$ is defined by the hyperspherical datum \[  (\GSp_2 \times \GSO_j)^{\text{sim}}/\mathbb{G}_m \times \SL_2 \rightarrow  \PGSO_{12-j} \times \PGSO_j \] (and trivial $S$), where the map into the $\PGSO_{12-j}$ factor is as usual (defined by $\gamma$) while the map into the $\PGSO_j$ factor is by projection. 

Similarly, when $j$ is odd, then $G_j = \PGSp_{11-j} \times \PGSp_{j-1}$. 

Let $\gamma$ be the nilpotent orbit of $\PGSp_{11-j}$ with partition $[6-j, 6-j, 1^{j-1}]$; then $M_\gamma = (\GSp_2 \times \GSp_{j-1})^{\text{sim}}/\mathbb{G}_m$. 

Then the hyperspherical dual $M_j$ is defined by the hyperspherical datum \[  (\GSp_2 \times \GSp_{j-1})^{\text{sim}}/\mathbb{G}_m \times \SL_2 \rightarrow  \PGSp_{11-j} \times \PGSp_{j-1} \] (and trivial $S$), where the map into the $\PGSp_{11-j}$ factor is as usual (defined by $\gamma$) while the map into the $\PGSp_{j-1}$ factor is by projection.

(In informal terms, while the cases in \cite{FU} considers the residual action of the entire $M_\gamma$, and the cases we have mainly considered so far have `reduced away' this residual action of $M_\gamma$, here we are considering the `partial' residual action of a \textit{factor} of $M_\gamma$.)
\end{thm2}

Note that when $j=3$ or $4$, we are essentially considering the nilpotent orbits with partitions $[3,3,1,1]$ and $[2,2,1,1,1,1]$ respectively, which have already featured in some of our other examples above. When $j=0,1,2$, we of course essentially recover the last three examples of Expectation \ref{ex:Exceptional} respectively. We will mention the $j=5$ case below. 
\vskip 5pt

On one hand, the branching problems associated to the quantization of $M_j^\vee$ can in principle be resolved by the exceptional theta correspondence for the dual pair \[ (\SL_2 \times_{\mu_2} \Spin_j, \Spin_{12-j}) \] in the exceptional group of type $E_7$. 

Indeed, continue the notation of \S \ref{66partition}; then the quantization of $M_j^\vee$ is the representation $\omega_\psi$ further restricted to $\Spin_{12-j}\times \Spin_j$. As in \S \ref{66partition}, one has: as a representation of $\Spin_{12-j}\times \Spin_j$,
\[ \omega_\psi\cong \Pi_{N,\psi},\] 
for $N$ a maximal unipotent subgroup of $\SL_2$. 

Suppose now Howe duality holds for the dual pair $(\SL_2 \times_{\mu_2} \Spin_j, \Spin_{12-j})$; then one has the irreducible quotients of $\Pi$ are in general of the form \[ \pi \boxtimes \sigma \boxtimes \Theta(\pi\boxtimes \sigma) \] for $\pi\in\Irr(\SL_2)$ and $\sigma\in\Irr(\Spin_j)$. 

It follows then that the irreducible quotients of $\omega_\psi\cong \Pi_{N,\psi}$ are representations of $\Spin_j\times \Spin_{12-j}$ consisting in general of lifts from $\SL_2\times_{\mu_2} \Spin_j$, and whose A-parameters should factor through the defining datum of $M_j$ \[ (\GSp_2 \times \GSO_j)^{\text{sim}}/\mathbb{G}_m \times \SL_2 \rightarrow  \PGSO_{12-j} \times \PGSO_j \] (if $j$ is even, and similarly if $j$ is odd), recalling also how this map was defined above.  
\vskip 5pt

On the other hand, the branching problems associated to the quantization of $M_j$ are of strongly-tempered type. As mentioned, we have already considered several cases above. Of particular note is the extreme case $j=5$, in which case the corresponding nilpotent orbit $\gamma$ is trivial, and $M_j$ would be (the cotangent bundle of) $\big((\GSp_2 \times \GSp_{4})^{\text{sim}}/\mathbb{G}_{m} \big)\bs \big(\PGSp_{6} \times \PGSp_{4}\big)$, which is a strongly tempered spherical variety \cite{WZ}.

\end{document}